\documentclass[letterpaper, 10 pt, conference]{IEEEtran}  % Comment this line out
\IEEEoverridecommandlockouts                              % This command is only
\usepackage{graphics} % for pdf, bitmapped graphics files
\usepackage{epsfig} % for postscript graphics files
\usepackage{mathptmx} % assumes new font selection scheme installed
\usepackage{times} % assumes new font selection scheme installed
\usepackage{amsmath} % assumes amsmath package installed
\usepackage{amssymb}  % assumes amsmath package installed
\usepackage{eucal}
\usepackage{graphicx}
\usepackage{cite}
\usepackage{caption}
\usepackage{subcaption}
\usepackage{epstopdf}
\usepackage{latexsym}
\usepackage[german,english]{babel}
\usepackage{todonotes}
\usepackage{soul}
\usepackage{color}

\usepackage{wrapfig}

\graphicspath{{./images/}}

\title{
\LARGE \bf
Modeling and analysis of Duhem hysteresis operators with butterfly loops
}

\author{M. A. Vasquez-Beltran, B. Jayawardhana, R. Peletier% <-this % stops a space
    \thanks{*This paper is based on research developed in the DSSC Doctoral Training Programme, co-funded through a Marie Skłodowska-Curie COFUND (DSSC 754315).} %*This work was supported by EU COFUND programme 2017-2021 and by NWO/NSO PIPP funding.}% <-this % stops a space
    \thanks{ $^{1}$M. A. Vasquez Beltran and B. Jayawardhana are with the Engineering and Technology Institute Groningen, Faculty of Science and Engineering, University of Groningen, 9747AG Groningen, The Netherlands {\tt\small \{m.a.vasquez.beltran;b.jayawardhana\} @rug.nl} }%
    \thanks{ $^{2}$R. Peletier is with the Kapteyn Astronomical Institute, Faculty of Science and Engineering, University of Groningen, 9747AG Groningen, The Netherlands {\tt\small r.peletier@rug.nl} }%
}

\newenvironment{proof}{\vspace{.1cm}\noindent{\sc Proof.   }\hspace{0.05cm}\,\,}{$\hfill\Box$\vspace{.1cm}}
\newtheorem{theorem}            {Theorem}[section]

\newtheorem{lemma}              [theorem]{Lemma} 
\newtheorem{proposition}		[theorem]{Proposition} 
\newtheorem{corollary}		[theorem]{Corollary}

\newcommand{\dd}{{\rm d}\hbox{\hskip 0.5pt}}

\newcommand{\rline}{{\mathbb R}}
\newcommand{\R}{{\mathbb R}}

\newcommand{\N}{{\mathbb N}}

\DeclareMathOperator*{\sign}{sign}

\begin{document}

\bstctlcite{BSTcontrol}
\pagestyle{empty}
\maketitle
\thispagestyle{empty}

%%%%%%%%%%%%%%%%%%%%%%%%%%%%%%%%%%%%%%%%%%%%%%%%%%%%%%%%%%%%%%%%%%%%%%%%%%%%%%%%

\begin{abstract}
    In this work we study and analyze a class of Duhem hysteresis operators that can exhibit butterfly loops. We study firstly the consistency property of such operator which corresponds to the existence of an attractive periodic solution when the operator is subject to a periodic input signal. Subsequently, we study the two defining functions of the Duhem operator such that the corresponding periodic solutions can admit a butterfly input-output phase plot. We present a number of examples where the Duhem butterfly hysteresis operators are constructed using two zero-level set curves that satisfy some mild conditions. 
    %which guarantee the existence of a periodic solution to the scalar version of the Duhem hysteresis operator, and we use this results to introduce an new class of Duhem operator whose gradient functions satisfy mild assumptions and whose output can portray butterfly hysteresis loops.
\end{abstract}

%\begin{IEEEkeywords} 
%    Mechatronics, Control applications, Iterative learning control
%\end{IEEEkeywords} 

%%%%%%%%%%%%%%%%%%%%%%%%%%%%%%%%%%%%%%%%%%%%%%%%%%%%%%%%%%%%%%%%%%%%%%%%%%%%%%%%
\section{INTRODUCTION}

Hysteresis is a natural phenomenon that was originally investigated in the study of %found affecting the relation between the 
magnetic field and magnetic flux density in ferromagnetic materials \cite{Ewing1882}. In the following centuries, the hysteresis phenomena are well-documented and studied in numerous systems originating from various disciplines, %Nowadays, hysteresis is a phenomenon that has been observed affecting 
%a wide variety of real-world processes. 
from biology \cite{Angeli2004,Noori2014}, physics \cite{Brokate1996}, economics \cite{Bakas2020} to experimental psychology \cite{Poltoratski2014}. The hysteresis is typically characterized by the presence of memory in its (dynamic) behaviour and has attracted the attention of scientists for its  intrinsic complexity.  % systems exhibiting hysteresis have been observed and described
%attracted the attention of science due to their intrinsic complexity. 
The multitude of domains, where hysteresis can be found, has led most of the works in literature to describe it using phenomenological models which are rather independent of the specific process underlying it. In this regard, the Duhem model \cite{Ikhouane2018} is one of the well-known  generic models of hysteresis. Its mathematical formulation encompasses many of other phenomenological models, for instance, the Dahl model, the Bouc-Wen model and the Maxwell-slip model \cite{Macki1993}. Another large class of popular models is the Preisach models \cite{Mayergoyz1986,Macki1993} which will not be considered in this paper. 
%\todo[inline]{Add a reference for each one of the models and maybe put reference to original Duhem model?}

The Duhem model has been extensively studied in the literature and several mathematical properties have been established. Roughly speaking, the Duhem model maps input signals to output signals via switched nonlinear differential equations, where the switch signal depends on the sign of the derivative of its input signal. 
%The Duhem operator, which is roughly speaking the mapping associating input functions to  output solutions that satisfy the equation of the Duhem model, 
Mathematical properties of the resulting Duhem operator (with time-independent vector fields) have been presented in literature that include the existence and uniqueness of the solutions, as well as, monotonicity, semigroup and rate-independent properties. %has been shown to satisfy characteristics of a hysteresis operator such as semi-group structure and rate-independency. 
%Furthermore, conditions that guarantee the existence and uniqueness of a solution to the integro-differential equations that compose the model have be established as well. 
A thorough exposition of these properties and other fundamental mathematical properties can be found in \cite{Visintin1994,Ikhouane2018,MohammadNaser2013}. 
Control systems properties, where Duhem  operator is feedback interconnected with other nonlinear systems, have been studied in literature.  %useful for the design of control strategies for systems that include elements with hysteretic behavior modeled by a Duhem operator have also been studied. 
For instance, the study of dissipativity in a class of Duhem operators is presented in %conditions to obtain dissipativity and the special case known as passivity have been given in 
\cite{Jayawardhana2009,Jayawardhana2011,Jayawardhana2012} where the associated storage functions and supply rate functions depend on the specific hysteresis loops obtained from the Duhem models. %the construction of storage functions has been related to the orientation of the hysteresis loops formed in the input-output phase plot.
%A challenging question that has been addressed in different works is to characterize the conditions such that the application of a periodic input to the Duhem operator is guaranteed to produce a periodic output such that the input-output phase plot exhibits a closed loop. 
In recent decades, attention has also been given to the convergent systems property \cite{Pavlov2005} or consistency property \cite{Ikhouane2013,Naser2013} of Duhem model where the output converges to a periodic signal when a periodic input signal is given. Such property in the literature of hysteresis is known as the accommodation property as presented for instance in \cite{Bree2009} which investigates the hysteresis modeling in ferromagnetic material. In this case, the phase plot of input and output signals will show loops that converge %In practical applications, the phase plot of a pair of system variables exhibiting a hysteretic relation is always observed to produce loops and in certain cases a phenonmenon known as accommodation can be found (see for instance in \cite{Bree2009}) which can be roughly described as the hysteresis curved produced in the phase plot converging 
to a limit cycle %closed loop oscillating 
around the so-called anhysteresis curve. This convergent systems property has been shown for  the semi-linear Duhem model \cite{Oh2005} and for the
%Sufficient conditions on the Duhem model operator is capable of exhibiting this behavior and conditions to guaranty the convergence of the output to a periodic function such that the input-output phase plot forms a closed loop have been determined for instance in \cite{Oh2005} for the semi-linear Duhem model and recently in \cite{Ikhouane2021} for the 
Babu$\check{\text{s}}$ka's model \cite{Ikhouane2021} which is a class of the Duhem model where each vector field in the %one of the gradient functions in the 
integro-differential equations can be %decomposed into two independent functions. %is, 
expressed as the multiplication of two single variable functions.

In this paper, we extend the aforementioned results to a class of Duhem model that can exhibit asymmetric butterfly loops. Here, the butterfly loops refer to presence of closed orbits with two or multiple loops in the input-output phase plot. While the standard hysteresis operators produce either counterclockwise or clockwise loops, the butterfly ones comprise of both clockwise and counterclockwise loops. The presence of butterfly loops has been shown and observed in practice for decades, e.g. in piezoactuator systems \cite{Gu2016} and in magnetostrictive materials \cite{dupre2003preisach}. The first simple mathematical modeling, analysis and identification of hysteresis with butterfly loops is presented in \cite{Drinvcic2011} where a convex function is added to the output of standard hysteresis operator in order to enforce two inflection points to the standard loop and thereby creating butterfly loops. A general modeling and analysis of butterfly hysteresis operator based on  Preisach model is presented in \cite{Vasquez2021} and is firstly reported in \cite{Jayawardhana2019} which is used to describe the shape memory property of a newly purposely-designed  piezoelectric materials. This framework has been used in the development of deformable mirror with high-density actuators  \cite{Schmerbauch2020,Huisman2020,Vasquez2021}. As far as the authors are aware of, the modeling and analysis of Duhem model that can exhibit butterfly loops remain largely open and it is the focus of this paper. 

%We follow from our previous work \cite{Vasquez2021}, where we have characterized classes of Preisach hysteresis operators capable of describing asymmetric butterfly hysteresis and multi-loop hysteresis behavior. We have called these operators the Preisach butterfly hysteresis operator and presented in \cite{Jayawardhana2019} their applicability describing the hysteretic relation exhibited between the voltage and strain of piezoelectric materials and in the actuators' modeling and control of a new concept of deformable mirror with high actuators density \cite{Schmerbauch2020,Huisman2020,Vasquez2021}.

As our first main contribution in the extension of previous results to the butterfly hysteresis operator using Duhem model, % In our first main contribution, Firstly, we present an 
%analysis where only the so-called 
we investigate the applicability of Babu$\check{\text{s}}$ka's conditions used in \cite{Ikhouane2021} as sufficient conditions for guaranteeing %are sufficient to ensure 
the convergence of the input-output phase plot to a closed orbit when the input signal is simple periodic\footnote{A periodic signal is called simple if it admits only one %an input with a single 
maximum and one minimum %a single minimum 
within its periodic interval.} in Section III. These conditions correspond to %can be simple put as 
the monotonicity of the vector fields in the Duhem model when the input argument is fixed.  % the vertical coordinate that compose the integro-differential equations of the scalar rate-independent Duhem model. 
Using only these  Babu$\check{\text{s}}$ka's conditions, we can relax the %usual additional restriction over these 
strong sign-definite assumption on these vector fields that are typically assumed in literature. %to be sign-definite. 
Furthermore, we show that if we have strict monotonicity conditions %these monotonicity with respect to the vertical coordinate is strict, 
then the closed orbit is unique for any initial value of the output. 
In Section IV, we present our  second main contribution where we study % we use the previous result to introduce 
a class of Duhem model whose vector fields are sign-indefinite but  satisfy the aforementioned Babu$\check{\text{s}}$ka's conditions. % and whose vector fields can take positive and negative values. 
Under some additional mild assumptions on the vector fields, we show that the input-output phase plot of this Duhem model converges to a closed orbit with two or more loops, e.g., it exhibits the butterfly loop. At the end of Section IV, we provide illustrative examples of this class of Duhem model. 

\section{PRELIMINARIES}\label{sec:preliminaries}

\textbf{Notation.} We denote by $C(U,Y)$, $AC(U,Y)$, $C_{\text{pw}}(U,Y)$ the spaces of continuous, absolute continuous, and piece-wise continuous functions $f: U\to Y$, respectively. We denote $\mathbb R_+:= [0,\infty)$. \\

We define the next two auxiliary operators which are used throughout this work.
The right-shift operator $\mathcal{S}_\tau\colon AC(\R_+,\R)\to AC(\R_+,\R)$ parameterized by $\tau\in\R$ is defined by
\begin{equation}\label{eq:shift_operator}
    \left[\mathcal{S}_\tau (v)\right](t) := v(t+\tau).
\end{equation}
The right-continuation operator $\mathcal{R}_\tau\colon AC(\R_+,\R)\to AC(\R_+,\R)$ parameterized by $\tau\in\R_+$ is defined by
\begin{equation}\label{eq:continuation_operator}
     \left[\mathcal{R}_\tau (v)\right](t) := \left\{\begin{array}{ll}
        v(t) & \text{if } t\in[0,\tau],\\
        v(\tau)    & \text{if } t\in(\tau,\infty).
    \end{array}\right.
\end{equation}

%%%%%%%%%%%%%%%%%%%%%%%%%%%%%%%%%%%%%
%We also define an hysteresis operator as follows
%\begin{definition}\label{def:time_transformation}
%    We call $\phi\colon\R_+\to\R_+$ a {\em time transformation} if $\phi(t)$ is continuous and increasing with $\phi(0)=0$ and $\lim_{t\to \infty}\phi(t)=\infty$. 
%\end{definition}
%\begin{definition}\label{def:rate_independent}
%    An operator $\Phi$ is said to be {\em rate independent} if 
%    \begin{equation*}
%        \big[\Phi(u\circ \phi)\big] = \big[\Phi(u)\big] \circ \phi
%    \end{equation*}
%    holds for all $u\in AC(\R_+,\R)$ and all admissible time transformation $\phi$.
%\end{definition}
%\begin{definition}\label{def:causal}
%    The operator $\Phi$ is said to be {\em causal} if for all $\tau>0$ and all $u_1,\,u_2\in AC(\R_+,\R)$ it holds that
%    \begin{multline*}
%        u_1(t)=u_2(t) \ \ \ \forall t\in [0,\tau] \\ 
%        \Rightarrow \big[\Phi(u_1)\big](t) = \big[\Phi(u_2)\big](t) \ \ \ \forall t\in[0,\tau].
%    \end{multline*} 
%\end{definition}
%\begin{definition}
%    If we let $y=\Phi(u,y_0)$
%    \begin{equation}
%        \mathcal{S}_\tau(\Phi(u,y_0)) = \Phi(\mathcal{S}_\tau(u),\mathcal{S}_\tau(y))
%    \end{equation}
%\end{definition}
%%%%%%%%%%%%%%%%%%%%%%

The Duhem hysteresis operator operator is a mapping $\Phi\colon AC(\R_+,\R)\times\R\to AC(\R_+,\R)$ such that $y=\Phi(u,y_0)$ satisfies
\begin{equation}\label{eq:duhem_model}
    \begin{aligned}
        \dot{y}(t) &= \left\{
            \begin{array}{ll}
                f_1(u(t),y(t))\dot{u},&\text{if }\dot{u}(t)\geq 0,\\
                f_2(u(t),y(t))\dot{u},&\text{if }\dot{u}(t)< 0,
            \end{array}
        \right.\\
        y(0)&=y_0,
    \end{aligned}
\end{equation}
at almost every $t\geq 0$ and with $f_1,f_2\in C^0(\R^2,\R)$.
Given an arbitrary input $u\in AC(\R_+,\R)$ and initial condition $y_0\in\R$, the existence and uniqueness of $y\in AC(\R_+,\R)$ satisfying \eqref{eq:duhem_model} at almost every $t\in[0,T]$ with $T>0$ is studied in \cite{Macki1993,Visintin1994} and guaranteed when $f_1$ and $f_2$ satisfy
\begin{align}
    \left( f_1(\upsilon,\gamma_1) - f_1(\upsilon,\gamma_2) \right)\left( \gamma_1 - \gamma_2 \right) &\leq \hphantom{-}\lambda_1 (u) (\gamma_1 - \gamma_2)^2, \\
    \left( f_2(\upsilon,\gamma_1) - f_2(\upsilon,\gamma_2) \right)\left( \gamma_1 - \gamma_2 \right) &\geq -\lambda_2 (u) (\gamma_1 - \gamma_2)^2,
\end{align}
for every $\upsilon,\gamma_1,\gamma_2\in\R$ and some for non-negative functions $\lambda_1,\lambda_2\in C(\R,\R_+)$.

An important property of the Duhem operator $\Phi$ as defined in \eqref{eq:duhem_model} is that it is rate-independent. In other words, for every $\phi\in C(\R_+,\R_+)$ such that $\phi(0)=0$, increasing and radially unbounded (i.e. $\phi(t)\to\infty$ as $t\to\infty$) we have
\begin{equation*}
    \big[\Phi(u\circ\phi,y_0)\big](t) = \left[\Phi(u,y_0)\circ\phi\right](t).
\end{equation*}

Moreover, following the work of  \cite{Visintin1994}, we consider hysteresis operator that satisfies the semi-group property, which means that if $y=\Phi(u,y_0)$ then
\begin{equation*}\label{eq:duhem_semigroup_pro}
    \mathcal{S}_\tau(\Phi(u,y_0)) = \Phi(\mathcal{S}_\tau(u),\mathcal{S}_\tau(y)).
\end{equation*}
%\todo[inline]{Add semi-group property, hysteresis operator definition, rate-independent definition, and other properties}

Throughout this work we assume that the implicit function $\upsilon\mapsto\{ \gamma\in\R\ |\ f_1(\upsilon,\gamma)-f_2(\upsilon,\gamma)=0 \}$ admits an explicit solution 
\begin{equation}\label{eq:anhysteresis_function}
    \gamma=\alpha(\upsilon)
\end{equation}
with $\alpha\in C^0(\mathbb R,\mathbb R)$, which we call the \emph{anhysteresis function} and  the corresponding curve (generated by $\alpha$) given by  %described by this function, given by 
\begin{equation}\label{eq:anhysteresis_curve}
    \mathcal{A} = \left\{ (\upsilon,\gamma)\in \R^2\ |\ \gamma=\alpha(\upsilon) \right\},
\end{equation}
is called the \emph{anhysteresis curve}. By definition, the curve $\mathcal A$ divides the input-output plane into two regions where %such that 
$f_1(\upsilon,\gamma_1)-f_2(\upsilon,\gamma_1)\geq0$ whenever $\gamma_1\geq\gamma=\alpha(\upsilon)$, and $f_1(\upsilon,\gamma_1)-f_2(\upsilon,\gamma_1)\leq0$ whenever $\gamma_1\leq\gamma=\alpha(\upsilon)$.

\section{THE DUHEM OPERATOR ACCOMMODATION PROPERTY}\label{sec:accommodation_property}

As briefly described in the Introduction, the accommodation property of the Duhem operator $\Phi$ %is, roughly speaking, 
refers to the property where %the characteristic of its 
the input-output phase plot always converges to %approach 
a periodic closed orbit when the input signal is periodic \cite{VanBree2009}. In this section, we formally study this property and prove that when the input is periodic with a single maximum and a single minimum in its periodic interval, the input-output phase plot approaches a unique periodic closed-loop which revolves in a neighborhood of the anhysteresis curve $\mathcal{A}$. We begin studying the input-output phase plot produced by the application of monotonic inputs and then we extend our analysis to periodic inputs.
For simplicity of notation, in what follows we use $Y_u\colon \R_+\times\R\to\R$, which we define by
\begin{equation*}\label{eq:Yu(t,y0)}
    Y_u(t,y_0) := \left[\Phi(u,y_0)\right](t),
\end{equation*}
to refer to the output of the Duhem operator $\Phi$ parameterized by the input signal $u$ and with the time $t$ and initial condition $y_0$ as independent variables.

\subsection{The Duhem operator with monotonic inputs}
Let $u_+\in AC(\R_+,\R)$ be an input which is monotonically increasing in $[0,\infty)$ and consider a sub-interval $[0,\tau_1]$ with $\tau_1>0$ such that $u(0)=\upsilon_{\min}<\upsilon_{\max}=u(\tau_1)$. Since the Duhem operator $\Phi$ defined with \eqref{eq:duhem_model} is rate-independent as shown in \cite{Visintin1994,Ikhouane2018,Oh2005}, for every $t\in[0,\tau_1]$ we have that 
\begin{equation}\label{eq:integral_Yu+}
    \begin{aligned}
        Y_{u_+}(t,y_0) - y_0 &= \int_{0}^{t} f_1 \left( u_+(\tau),\ Y_{u_+}(\tau,y_0) \right)\ \dot{u}(\tau)\ \dd\tau  \\
        &= \int_{\upsilon_{\min}}^{u_+(t)} f_1\left(\upsilon, \mathcal{Y}_{u_+}(\upsilon,y_0) \right)\ \dd\upsilon\\
        &= \mathcal{Y}_{u_+}(u_+(t),y_0) - \mathcal{Y}_{u_+}(\upsilon_{\min},y_0)
    \end{aligned}
\end{equation}
where 
\begin{equation*}
    \mathcal{Y}_{u_+}\colon[\upsilon_{\min}, \upsilon_{\max}]\times\R\to\R
\end{equation*}
is the parameterization of the corresponding solution $Y_{u_+}(t,y_0)$ with the instantaneous value of the input $u_+$ and the initial condition $y_0$ as independent variables (i.e. $\mathcal{Y}_{u_+}(\upsilon(t),y_0)=Y_{u_+}(t,y_0)$ with $\upsilon(t) = u_+(t)$ for every $t\in[0,\tau_1]$).

Analogously, let $u_-\in AC(\R_+,\R)$ be an input which is monotonically decreasing in $[0,\infty)$ and consider a sub-interval $[0,\tau_2]$ with $\tau_2>0$ such that $u(0)=\upsilon_{\max}>\upsilon_{\min}=u(\tau_2)$. By the rate-independent property of the Duhem operator, we have that, for every $t\in[0,\tau_2]$,  
\begin{equation}\label{eq:integral_Yu-}
    \begin{aligned}
        Y_{u_-}(t,y_0) - y_0 &= \int_{0}^{t} f_2 \left( u_-(\tau),\ Y_{u_-}(\tau,y_0) \right)\ \dot{u}(\tau)\ \dd\tau  \\
        &= \int_{\upsilon_{\max}}^{u_-(t)} f_2\left(\upsilon, \mathcal{Y}_{u_-}(\upsilon,y_0) \right)\ \dd\upsilon\\
        &= \mathcal{Y}_{u_-}(u_-(t),y_0) - \mathcal{Y}_{u_-}(\upsilon_{\max},y_0)
    \end{aligned}
\end{equation}
where in this case
\begin{equation*}
    \mathcal{Y}_{u_-}\colon[\upsilon_{\min}, \upsilon_{\max}]\times\R\to\R
\end{equation*}
is the parameterization of the corresponding solution $Y_{u_-}(t,y_0)$ with the instantaneous value of the input $u_-$ and the initial condition $y_0$ as independent variables (i.e. $\mathcal{Y}_{u_-}(\upsilon,y_0)=Y_{u_-}(t,y_0)$ with $\upsilon = u_-(t)$ for every $t\in[0,\tau_2]$).

In what follows, we present a series of auxiliary lemmas necessary to prove the accommodation property. Firstly, using the parameterizations $\mathcal{Y}_{u_+}$ and $\mathcal{Y}_{u_-}$ of the output, we state formally in the first lemma that two input-output phase plots obtained with the same monotonic input but from different initial conditions never cross each other.

\vspace*{2mm}
\begin{lemma}\label{lemma:non_intersecting_calYu}
    The next statements are true.
    \begin{itemize}
        \item[\emph{a)}] If two initial conditions satisfy $\gamma_a\leq \gamma_b$, then we have
        \begin{equation*}
            \begin{aligned}
                \mathcal{Y}_{u_+}(\upsilon,\gamma_a)&\leq\mathcal{Y}_{u_+}(\upsilon,\gamma_b), \\
                \mathcal{Y}_{u_-}(\upsilon,\gamma_a)&\leq\mathcal{Y}_{u_-}(\upsilon,\gamma_b),
            \end{aligned}
        \end{equation*}
        for every $\upsilon\in[\upsilon_{\min},\upsilon_{\max}]$.\vspace*{1mm}

        \item[\emph{b)}] If we have that
        \begin{equation*}
            \begin{array}{r@{}l}
                                   & \mathcal{Y}_{u_+}(\upsilon_{\max},\gamma_a) < \mathcal{Y}_{u_+}(\upsilon_{\max},\gamma_b) \\
               \big(\text{resp.  } & \mathcal{Y}_{u_-}(\upsilon_{\min},\gamma_a) < \mathcal{Y}_{u_-}(\upsilon_{\min},\gamma_b) \big),
            \end{array}
        \end{equation*}
        then
        \begin{equation*}
            \begin{array}{r@{}l}
                                   & \mathcal{Y}_{u_+}(\upsilon,\gamma_a) < \mathcal{Y}_{u_+}(\upsilon,\gamma_b) \\
               \big(\text{resp.  } & \mathcal{Y}_{u_-}(\upsilon,\gamma_a) < \mathcal{Y}_{u_-}(\upsilon,\gamma_b) \big),
            \end{array}
        \end{equation*}
        for every $\upsilon\in[\upsilon_{\min},\upsilon_{\max}]$.
        \end{itemize}
\end{lemma}
\vspace*{1mm}
\begin{proof}
    We prove both statements for $\mathcal{Y}_{u_+}$ by contradiction as follows.\vspace*{1mm}
    
        \begin{figure}
        \centering
        \begin{subfigure}{0.50\linewidth}
            \centering
            \includegraphics[width=1.0\linewidth]{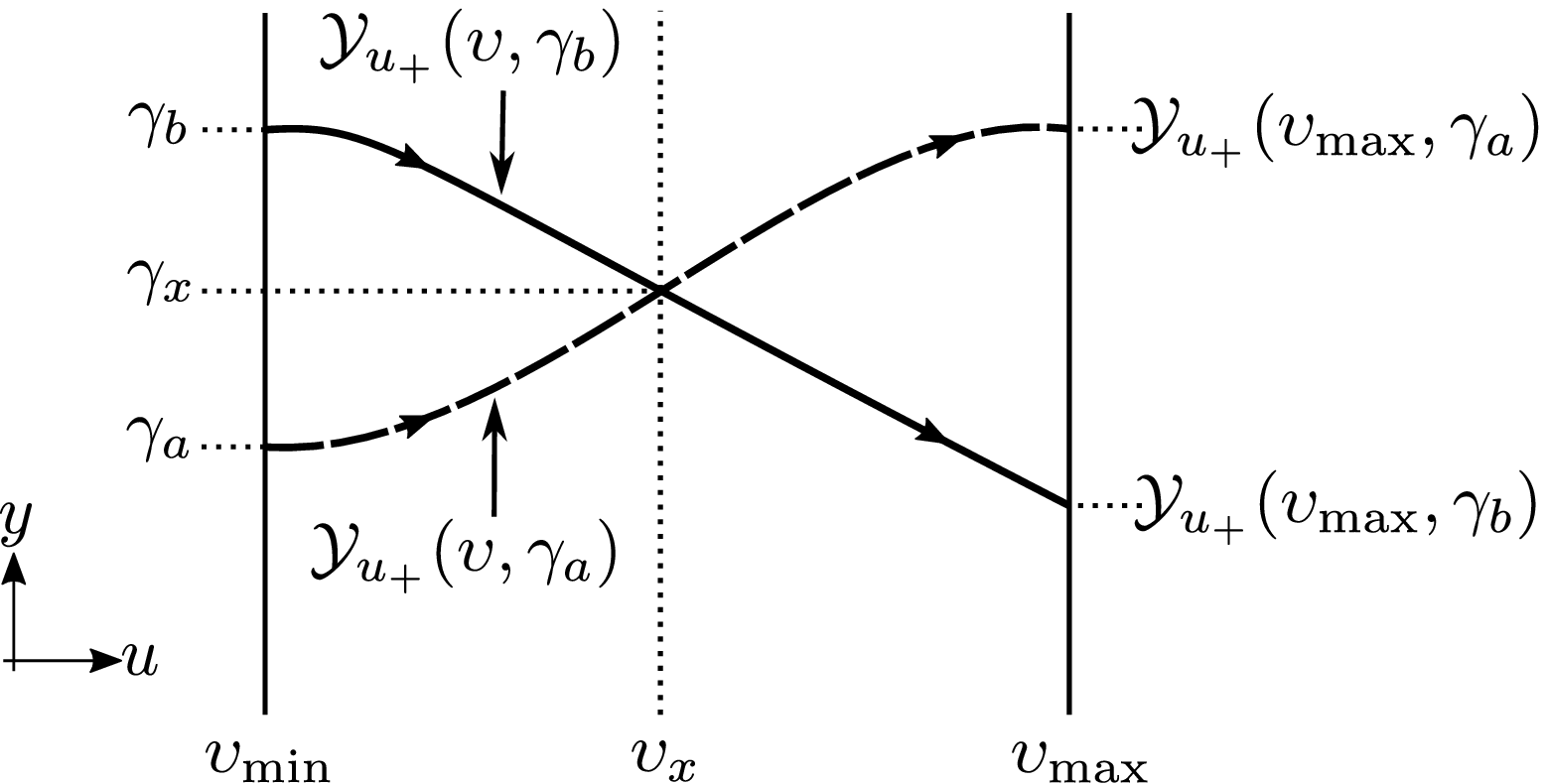}
            \caption[]{Contradiction of case \emph{a)} in Lemma \ref{lemma:non_intersecting_calYu}. \label{fig:intersecting_y+_a}}
        \end{subfigure}%
        \hspace*{2mm}%
        \begin{subfigure}{0.50\linewidth}
            \centering
            \includegraphics[width=1.0\linewidth]{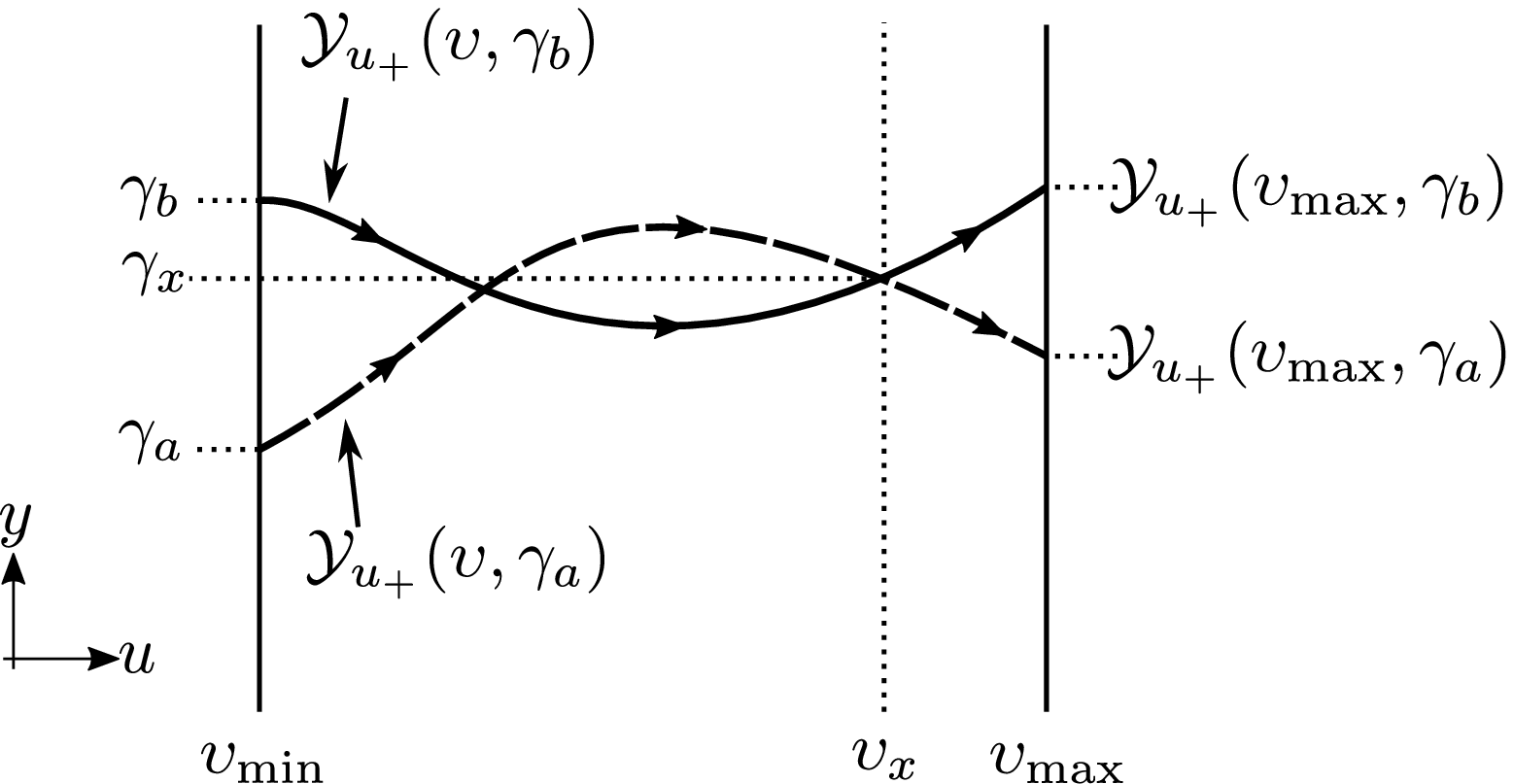}
            \caption[]{Contradiction of case \emph{b)} in Lemma \ref{lemma:non_intersecting_calYu}. \label{fig:intersecting_y+_b}}
        \end{subfigure}
        \caption{Illustration of non-possible intersection between solutions in the input-output phase plot as shown in Lemma \ref{lemma:non_intersecting_calYu} corresponding to $\mathcal{Y}_{u_+}$. The one for $\mathcal{Y}_{u_-}$ follows in a similar fashion.}
    \end{figure}
    
    \noindent \emph{Part a) } Let $\gamma_a\leq \gamma_b$ and assume that 
    $\mathcal{Y}_{u_+}(\upsilon_c,\gamma_a)>\mathcal{Y}_{u_+}(\upsilon_c,\gamma_b)$ for some $\upsilon_c\in[\upsilon_{\min},\upsilon_{\max}]$.\vspace*{1mm}
    
    \noindent Let $\tau_c\in [0,\tau_1)$ be the corresponding time such that $\mathcal{Y}_{u_+}(\upsilon_c,\gamma_a)=Y_{u_+}(\tau_c,\gamma_a)>Y_{u_+}(\tau_c,\gamma_b)=\mathcal{Y}_{u_+}(\upsilon_c,\gamma_b)$.\vspace*{1mm}
    
    \noindent By continuity of $\mathcal{Y}_{u_+}$, there exists $\upsilon_x\in[\upsilon_{\min},\upsilon_c)$ such that 
    $\gamma_x=\mathcal{Y}_{u_+}(\upsilon_x,\gamma_a)=\mathcal{Y}_{u_+}(\upsilon_x,\gamma_b)$ (see Fig. \ref{fig:intersecting_y+_a}).
    Let $\tau_x\in [0,\tau_c)$ be the corresponding time instance such that 
    $\gamma_x=Y_{u_+}(\tau_x,\gamma_a)=Y_{u_+}(\tau_x,\gamma_b)$. \vspace*{1mm}
    
    \noindent We can create a right-shifted input $u_s=\mathcal{S}_{\tau_x}(u_+)$ 
    and note that by the semi-group property of the Duhem operator we must have that
    \begin{equation*}
        Y_{u_+}(t+\tau_x,\gamma_a) = Y_{u_s}(t,\gamma_x) = Y_{u_+}(t+\tau_x,\gamma_b)
    \end{equation*}
    for every $t\in[0,\tau_1-\tau_x]$, which implies a contradiction to the uniqueness of solution $Y_{u_s}$ since
    \begin{equation*}
        Y_{u_+}(\tau_c,\gamma_a)>Y_{u_+}(\tau_c,\gamma_b).
    \end{equation*}
    Therefore, $\mathcal{Y}_{u_+}(\upsilon,\gamma_a)\leq\mathcal{Y}_{u_+}(\upsilon,\gamma_b)$ for every $\upsilon\in[\upsilon_{\min},\upsilon_{\max}]$.\vspace*{1mm}
    
    \noindent \emph{Part b) } Let $\mathcal{Y}_{u_+}(\upsilon_{\max},\gamma_a)<\mathcal{Y}_{u_+}(\upsilon_{\max},\gamma_b)$ and assume that 
    $\gamma_x=\mathcal{Y}_{u_+}(\upsilon_x,\gamma_a)=\mathcal{Y}_{u_+}(\upsilon_x,\gamma_b)$ for some $\upsilon_x\in[\upsilon_{\min},\upsilon_{\max})$ (see Fig. \ref{fig:intersecting_y+_b}).
    Letting $\tau_x\in [0,\tau_1)$ be the corresponding time instance such that 
    $\gamma_x=Y_{u_+}(\tau_x,\gamma_a)=Y_{u_+}(\tau_x,\gamma_b)$ and creating right-shifted input $u_s=\mathcal{S}_{\tau_x}(u_+)$ we can obtain the same contradiction as in \emph{Part a)}. Therefore, $\mathcal{Y}_{u_+}(\upsilon,\gamma_a)<\mathcal{Y}_{u_+}(\upsilon,\gamma_b)$ for every $\upsilon\in[\upsilon_{\min},\upsilon_{\max}]$\\

    We prove the statements for $\mathcal{Y}_{u_-}$ also by contradiction. \vspace*{1mm}
    
    \noindent \emph{Part a) } Let $\gamma_a\leq \gamma_b$ and assume that 
    $\mathcal{Y}_{u_-}(\upsilon_c,\gamma_a)>\mathcal{Y}_{u_-}(\upsilon_c,\gamma_b)$ for some $\upsilon_c\in[\upsilon_{\min},\upsilon_{\max}]$.\vspace*{1mm}
    
    \noindent Let $\tau_c\in [0,\tau_2)$ be the corresponding time such that $\mathcal{Y}_{u_-}(\upsilon_c,\gamma_a)=Y_{u_-}(\tau_c,\gamma_a)>Y_{u_-}(\tau_c,\gamma_b)=\mathcal{Y}_{u_-}(\upsilon_c,\gamma_b)$.\vspace*{1mm}
    
    \noindent By continuity of $\mathcal{Y}_{u_-}$, there exists $\upsilon_x\in(\upsilon_c,\upsilon_{\max})$ such that 
    $\gamma_x=\mathcal{Y}_{u_-}(\upsilon_x,\gamma_a)=\mathcal{Y}_{u_-}(\upsilon_x,\gamma_b)$.
    Let $\tau_x\in [0,\tau_c)$ be the corresponding time instance such that 
    $\gamma_x=Y_{u_-}(\tau_x,\gamma_a)=Y_{u_-}(\tau_x,\gamma_b)$.\vspace*{1mm}
    
    \noindent We can create again right-shifted input $u_s=\mathcal{S}_{\tau_x}(u_-)$ 
    and note that by 
    the semi-group property 
    of the Duhem operator we must have that
    \begin{equation*}
        Y_{u_-}(t+\tau_x,\gamma_a) = Y_{u_s}(t,\gamma_x) = Y_{u_-}(t+\tau_x,\gamma_b)
    \end{equation*}
    for every $t\in[0,\tau_2-\tau_x]$, which implies a contradiction to the uniqueness of solution $Y_{u_s}$ since
    \begin{equation*}
        Y_{u_-}(\tau_c,\gamma_a)>Y_{u_-}(\tau_c,\gamma_b).
    \end{equation*}
    Therefore, $\mathcal{Y}_{u_-}(\upsilon,\gamma_a)\leq\mathcal{Y}_{u_-}(\upsilon,\gamma_b)$ for every $\upsilon\in[\upsilon_{\min},\upsilon_{\max}]$. \vspace*{1mm}
    
    \noindent \emph{Part b) } Let $\mathcal{Y}_{u_-}(\upsilon_{\min},\gamma_a)<\mathcal{Y}_{u_-}(\upsilon_{\min},\gamma_b)$ and assume that 
    $\gamma_x=\mathcal{Y}_{u_-}(\upsilon_x,\gamma_a)=\mathcal{Y}_{u_-}(\upsilon_x,\gamma_b)$ for some $\upsilon_x\in(\upsilon_{\min},\upsilon_{\max}]$. 
    Letting $\tau_x\in [0,\tau_2)$ be the corresponding time instance such that 
    $\gamma_x=Y_{u_-}(\tau_x,\gamma_a)=Y_{u_-}(\tau_x,\gamma_b)$ and creating right-shifted input $u_s=\mathcal{S}_{\tau_x}(u_-)$ we can obtain the same contradiction as in \emph{Part a)} for $\mathcal{Y}_{u_-}$.  Therefore, $\mathcal{Y}_{u_-}(\upsilon,\gamma_a)<\mathcal{Y}_{u_-}(\upsilon,\gamma_b)$ for every $\upsilon\in[\upsilon_{\min},\upsilon_{\max}]$. 
\end{proof}

\vspace*{1mm}

\subsection{The Duhem operator with a periodic input}
We can now analyze the behavior of the Duhem operator $\Phi$ when the applied input signal is simple periodic. For this, let $u_p\in AC(\R_+,\R)$ be a periodic input with period $T>0$ and with one minimum and one maximum $\upsilon_{\min}<\upsilon_{\max}$ in its periodic interval. Without loss of generality, we assume that  $u_p(0)=\upsilon_{\min}$ and $u_p(t_1)=\upsilon_{\max}$ for some $t_1\in(0,T)$. In other words, $0<t_1<T$ is a monotonic partition of $[0,T]$.
We can split $u_p$ into its two monotonic intervals using the right-shift and right-continuation operators \eqref{eq:shift_operator} and \eqref{eq:continuation_operator}, which are formalized using two functions $u_{p+},u_{p-}\in AC(\R_+,\R)$ given by
\begin{equation*}
    \begin{aligned}
        u_{p+} &= \mathcal{R}_{t_1}(u_p),\\
        u_{p-} &= \mathcal{S}_{t_1}(\mathcal{R}_T(u_p)),
    \end{aligned}
\end{equation*}
whose corresponding outputs when applied to the Duhem operator are given by $Y_{u_{p+}}(t,\gamma)$ and $Y_{u_{p-}}(t,\zeta)$ for some initial conditions $\gamma,\zeta\in\R$.

Following the same argumentation as  before to obtain \eqref{eq:integral_Yu+} and \eqref{eq:integral_Yu-}, we can parameterize $Y_{u_{p+}}(t,\gamma)$ and $Y_{u_{p-}}(t,\zeta)$ by two mappings
\begin{align*}
    \mathcal{Y}_{u_{p+}}&\colon [\upsilon_{\min},\upsilon_{\max}]\times\R\to\R,\\ %\label{eq:calYup+},\\
    \mathcal{Y}_{u_{p-}}&\colon [\upsilon_{\min},\upsilon_{\max}]\times\R\to\R, %\label{eq:calYup-}
\end{align*}
respectively, where the instantaneous values of the inputs $u_{p+}$ and $u_{p-}$, and initial conditions $\gamma$ and $\zeta$ are the independent variables.

For arbitrary $\gamma_0\in\R$, let us define two sequences $(\zeta_n)_{n\in\N_0}$ and $(\gamma_n)_{n\in\N_0}$ recursively by
\begin{align}
    \zeta_n      &:= \mathcal{Y}_{u_{p+}}(\upsilon_{\max},\gamma_n),\label{eq:zeta_n}\\
    \gamma_{n+1} &:= \mathcal{Y}_{u_{p-}}(\upsilon_{\min},\zeta_n)  \label{eq:gamma_n+1}.
\end{align}
Note then that making $\gamma_0=y_0$, the output $Y_{u_p}(t,y_0)$ can be constructed recursively by
\begin{equation*}
    Y_{u_p}(t,y_0) = \left\{\begin{array}{l r@{} r@{} c@{} l}
        \mathcal{Y}_{u_{p+}}(u_p(t),\gamma_n), & \text{if } &     nT&\leq t<&t_1+nT, \\
        \mathcal{Y}_{u_{p-}}(u_p(t),\zeta_n),  & \text{if } & t_1+nT&\leq t<&(n+1)T,
    \end{array}\right.
\end{equation*}
with $n\in\N_0$. Therefore, we study the convergence of the solution $Y_{u_p}$ to a periodic solution using this sequences.

The next three lemmas present properties of the sequences $(\zeta_n)_{n\in\N_0}$ and $(\gamma_n)_{n\in\N_0}$ generated by the recursive composition of the function $\mathcal{Y}_{u_{p+}}$ and $\mathcal{Y}_{u_{p-}}$ that will be used in the proof of the main result of this section.

\vspace*{2mm}
\begin{lemma}\label{lemma:sequences_monotonic}
    Let $\gamma_0\in\R$. The sequences $(\zeta_n)_{n\in\N_0}$ and $(\gamma_n)_{n\in\N_0}$ generated by \eqref{eq:zeta_n} and \eqref{eq:gamma_n+1} are monotonic in the same direction (i.e. both increasing or both decreasing).\\
\end{lemma}
\vspace*{1mm}
\begin{proof}
    By induction, let $\gamma_i\geq\gamma_{i+1}$ and note that by Lemma \ref{lemma:non_intersecting_calYu}, we have
    \begin{equation*}
        \begin{aligned}            
            \zeta_i=\mathcal{Y}_{u_{p+}}(\upsilon_{\max},\gamma_i)&\geq\mathcal{Y}_{u_{p+}}(\upsilon_{\max},\gamma_{i+1})=\zeta_{i+1}\\
            \gamma_i=\mathcal{Y}_{u_{p-}}(\upsilon_{\min},\zeta_i), &\geq\mathcal{Y}_{u_{p-}}(\upsilon_{\min},\zeta_{i+1})=\gamma_{i+2},
        \end{aligned}
    \end{equation*}
    which proves that both sequences increasing. 
    Reversing all previous inequalities proves that both sequences are decreasing.
\end{proof}
\vspace*{1mm}

\vspace*{2mm}
\begin{lemma}\label{lemma:sequences_constant}
    Let $\gamma_0\in\R$ and consider the sequences $(\zeta_n)_{n\in\N_0}$ and $(\gamma_n)_{n\in\N_0}$ generated by \eqref{eq:zeta_n} and \eqref{eq:gamma_n+1}.
    Then the next statements are true.
    \begin{itemize}
        \item[a)] If $\gamma_i=\gamma_{i+1}$ for some $i\in\N_0$ then for every $k\geq i$ we have
        \begin{equation*}
            \begin{aligned}
                \zeta_{k}=\zeta_{k+1} \qquad\text{and}\qquad \gamma_{k+1}=\gamma_{k+2}.
            \end{aligned}
        \end{equation*}
        \item[b)] If $\zeta_j=\zeta_{j+1}$ for some $j\in\N_0$ then for every $k\geq j$ we have
        \begin{equation*}
            \begin{aligned}
                \gamma_{k+1}=\gamma_{k+2} \qquad\text{and}\qquad \zeta_{k+1}=\zeta_{k+2}.
            \end{aligned}
        \end{equation*}    
    \end{itemize}
\end{lemma}
\vspace*{1mm}
\begin{proof}
    Let $\gamma_i=\gamma_{i+1}$ and note that the uniqueness of solution $Y_{u_{p+}}$ implies
    $\mathcal{Y}_{u_{p+}}(\upsilon,\gamma_i) = \mathcal{Y}_{u_{p+}}(\upsilon,\gamma_{i+1})$ for every $\upsilon\in[\upsilon_{\min},\upsilon_{\max}]$,
    and consequently $\zeta_i=\mathcal{Y}_{u_{p+}}(\upsilon_{\max},\gamma_i) = \mathcal{Y}_{u_{p+}}(\upsilon_{\max},\gamma_{i+1})=\zeta_{i+1}$.
    Therefore, we have that
    \begin{equation*}
        \gamma_i=\gamma_{i+1} %\text{ implies }
        \qquad \Rightarrow \qquad \zeta_i=\zeta_{i+1}.
    \end{equation*}
    Similarly, when $\zeta_j=\zeta_{j+1}$, the uniqueness of solution $Y_{u_{p-}}$ implies $\mathcal{Y}_{u_{p-}}(\upsilon,\zeta_j) = \mathcal{Y}_{u_{p-}}(\upsilon,\zeta_{j+1})$ for every $\upsilon\in[\upsilon_{\min},\upsilon_{\max}]$,
    and consequently $\gamma_{j+1}=\mathcal{Y}_{u_{p+}}(\upsilon_{\min},\zeta_j) = \mathcal{Y}_{u_{p-}}(\upsilon_{\min},\zeta_{j+1})=\gamma_{j+2}$.
    Thus we have that
    \begin{equation*}
        \zeta_j=\zeta_{j+1} %\text{ implies }
        \qquad \Rightarrow \qquad \gamma_{j+1}=\gamma_{j+2}.
    \end{equation*}
    It follows that combining both implications proves both statements.
\end{proof}
\vspace*{1mm}

\vspace*{2mm}
\begin{lemma}\label{lemma:sequences_unbounded}
    Let $\gamma_0\in\R$ and consider the sequences $(\zeta_n)_{n\in\N_0}$ and $(\gamma_n)_{n\in\N_0}$ generated by \eqref{eq:zeta_n} and \eqref{eq:gamma_n+1}.
    The sequence $(\zeta_n)_{n\in\N_0}$ is unbounded if and only if $(\gamma_n)_{n\in\N_0}$ is unbounded. Moreover, when they are unbounded, they are strictly monotonic in the same direction (i.e. both strictly increasing or both strictly decreasing).
\end{lemma}
\vspace*{1mm}
\begin{proof}
    To prove the if part, let $(\gamma_n)_{n\in\N_0}$ be unbounded and note that by Lemma \ref{lemma:sequences_monotonic} we have that both sequences $(\gamma_n)_{n\in\N_0}$ and $(\zeta_n)_{n\in\N_0}$ are monotonic in the same direction. Moreover, by Lemma \ref{lemma:sequences_constant}, assuming that $\gamma_i=\gamma_{i+1}$ or $\zeta_j=\zeta_{j+1}$ for some $i,j\in\N_0$ implies that $(\gamma_n)_{n\in\N_0}$ is not unbounded, which is a contradiction. Therefore $(\zeta_n)_{n\in\N_0}$ is also unbounded and both are strictly monotonic.
    
    To prove the only if part, let $(\zeta_n)_{n\in\N_0}$ be unbounded and note that by Lemma \ref{lemma:sequences_monotonic} we have that both sequences $(\gamma_n)_{n\in\N_0}$ and $(\zeta_n)_{n\in\N_0}$ are monotonic in the same direction. Moreover, by Lemma \ref{lemma:sequences_constant}, assuming that $\gamma_i=\gamma_{i+1}$ or $\zeta_j=\zeta_{j+1}$ for some $i,j\in\N_0$ implies that $(\zeta_n)_{n\in\N_0}$ is not unbounded, which is a contradiction. Therefore $(\gamma_n)_{n\in\N_0}$ is also unbounded and both are strictly monotonic.
\end{proof}
\vspace*{1mm}

In the next pair of propositions, we introduce the main results of this section where sufficient conditions are presented such that the sequences $(\gamma_n)_{n\in\N_0}$ and $(\zeta_n)_{n\in\N_0}$ generated by \eqref{eq:zeta_n} and \eqref{eq:gamma_n+1} are convergent. We remark that if the sequences $(\zeta_n)_{n\in\N_0}$ and $(\gamma_n)_{n\in\N_0}$ are convergent to some pair $\gamma_*\in\R$ and $\zeta_*\in\R$, respectively, then due to the continuity and uniqueness of solution of the Duhem operator we must have that
\begin{equation*}\label{eq:periodicity_condition}
    \begin{aligned}
        \mathcal{Y}_{u_{p+}}(\upsilon_{\max},\gamma_*)&=\zeta_*,\\
        \mathcal{Y}_{u_{p-}}(\upsilon_{\min},\zeta_*)&=\gamma_*,
    \end{aligned}
\end{equation*}
and consequently both parameterized solutions $\mathcal{Y}_{u_{p+}}(\upsilon,\gamma_*)$ and $\mathcal{Y}_{u_{p-}}(\upsilon,\zeta_*)$ form a periodic closed orbit in the phase plot. 
With the first proposition we present a pair of inequalities that ensure the convergence to some periodic orbit. These inequalities have been previously presented in \cite{Ikhouane2021} and used together with other set of conditions to prove the convergence of the output to a periodic function for a specific version of the Duhem model known as Babu$\check{\text{s}}$ka's model. We show that only these two conditions are sufficient to ensure the convergence of the output to a periodic function in the scalar rate-independent Duhem model. Subsequently, with the second proposition we show that the strict versions of the inequalities ensure the uniqueness of the pair $\gamma_*\in\R$ and $\zeta_*\in\R$ and consequently the uniqueness of the closed periodic orbit.

\vspace*{2mm}
\begin{proposition}\label{prop:sequences_bounded_convergent}
    If the functions $f_1$ and $f_2$ in \eqref{eq:duhem_model} satisfy
    \begin{align}
        \left( f_1(\upsilon,\gamma_1) - f_1(\upsilon,\gamma_2) \right)\left( \gamma_1 - \gamma_2 \right) &\leq 0, \label{eq:f1_monotonic_convergent}\\
        \left( f_2(\upsilon,\gamma_1) - f_2(\upsilon,\gamma_2) \right)\left( \gamma_1 - \gamma_2 \right) &\geq 0, \label{eq:f2_monotonic_convergent}
    \end{align}
    for every $\gamma_1\neq\gamma_2$ and $\upsilon\in\R$,
    then for every $\gamma_0\in\R$ 
    the sequences $(\zeta_n)_{n\in\N_0}$ and $(\gamma_n)_{n\in\N_0}$ generated by \eqref{eq:zeta_n} and \eqref{eq:gamma_n+1} are convergent.
\end{proposition}
\vspace*{1mm}
\begin{proof}
    It follows from \eqref{eq:integral_Yu+} and \eqref{eq:integral_Yu-} that the difference between two consecutive elements in %the sequence 
    $(\gamma_n)_{n\in\N_0}$ is given by
    \begin{equation}\label{eq:gamma_i+1-gamma_i}
        \gamma_{i+1}-\gamma_i
        = \int_{\upsilon_{\min}}^{\upsilon_{\max}} \Big\{ 
            f_1\left( \upsilon,\mathcal{Y}_{u_{p+}}\left(\upsilon,\gamma_i\right) \right) 
          - f_2\left( \upsilon,\mathcal{Y}_{u_{p-}}\left(\upsilon,\zeta_i\right) \right)
            \Big\}\ \dd\upsilon
    \end{equation}
    Moreover, since by the definition of anhysteresis function $\alpha$, we have $f_1(\upsilon,\alpha(\upsilon))=f_2(\upsilon,\alpha(\upsilon))$ for every $\upsilon\in[\upsilon_{\min},\upsilon_{\max}]$, then we can add and subtract these terms inside the integral and obtain
    \begin{equation}\label{eq:gamma_i+1-gamma_i+-alpha}
        \begin{aligned}
            \gamma_{i+1}-\gamma_i
            &=\int_{\upsilon_{\min}}^{\upsilon_{\max}} \Big\{ 
                f_1\left( \upsilon,\mathcal{Y}_{u_{p+}}\left(\upsilon,\gamma_i\right) \right) 
              - f_1\left( \upsilon,\alpha(\upsilon) \right) 
                \Big\}\ \dd\upsilon \\
            &\hphantom{=}-\int_{\upsilon_{\min}}^{\upsilon_{\max}} \Big\{ 
                f_2\left( \upsilon,\mathcal{Y}_{u_{p-}}\left(\upsilon,\zeta_i\right) \right) 
              - f_2\left( \upsilon,\alpha(\upsilon) \right) 
              \Big\}\ \dd\upsilon.
        \end{aligned}
    \end{equation}
    %In what follows, we use the previous expression to prove this proposition. 
    
    We prove the proposition by contradiction. Assume that any of the sequences $(\zeta_n)_{n\in\N_0}$ or $(\gamma_n)_{n\in\N_0}$ is not convergent. Thus by Lemmas \ref{lemma:sequences_monotonic}, \ref{lemma:sequences_constant} and \ref{lemma:sequences_unbounded} both are unbounded and strictly monotonic in the same direction.
    
    On the one hand, if both are strictly increasing, then by Lemma \ref{lemma:non_intersecting_calYu} we can find two pairs $\gamma_i<\gamma_{i+1}$ and $\zeta_i<\zeta_{i+1}$ such that both $\mathcal{Y}_{u_{p+}}(\upsilon,\gamma_i)$ and $\mathcal{Y}_{u_{p-}}(\upsilon,\zeta_i)$ lie completely above the anhysteresis curve $\mathcal{A}$ given in \eqref{eq:anhysteresis_curve} (see Fig. \ref{fig:gamma_unbounded_above}). In other words, we have
    \begin{equation*}
        \mathcal{Y}_{u_{p+}}(\upsilon,\gamma_i)>\alpha(\upsilon)
        \quad\text{ and }\quad
        \mathcal{Y}_{u_{p-}}(\upsilon,\zeta_i)>\alpha(\upsilon),
    \end{equation*}
    for every $\upsilon\in[\upsilon_{\min},\upsilon_{\max}]$ and some $i\in\N_0$, where $\alpha$ is the anhysteresis function \eqref{eq:anhysteresis_function}. It follows from \eqref{eq:f1_monotonic_convergent} and \eqref{eq:f2_monotonic_convergent} that we have
    \begin{align*}
        f_1\left( \upsilon,\mathcal{Y}_{u_{p+}}\left(\upsilon,\gamma_i\right) \right) 
                - f_1\left( \upsilon,\alpha(\upsilon) \right) &\leq 0,\\
        f_2\left( \upsilon,\mathcal{Y}_{u_{p-}}\left(\upsilon,\zeta_i\right) \right) 
                - f_2\left( \upsilon,\alpha(\upsilon) \right) &\geq 0,
    \end{align*}
    for every $\upsilon\in[\upsilon_{\min},\upsilon_{\max}]$. 
    Consequently the right term of \eqref{eq:gamma_i+1-gamma_i+-alpha} is negative or zero, which is a contradiction since by the assumption the sequence is strictly increasing and $\gamma_{i+1}-\gamma_i>0$ for every $i\in\N_0$.
    
    On the other hand, if both sequences $(\zeta_n)_{n\in\N_0}$ and $(\gamma_n)_{n\in\N_0}$ are strictly decreasing, 
    then also by Lemma \ref{lemma:non_intersecting_calYu} we can find two pairs $\gamma_i>\gamma_{i+1}$ and $\zeta_i>\zeta_{i+1}$ such that both $\mathcal{Y}_{u_{p+}}(\upsilon,\gamma_i)$ and $\mathcal{Y}_{u_{p-}}(\upsilon,\zeta_i)$ lie completely below the anhysteresis curve $\mathcal{A}$ (see Fig. \ref{fig:gamma_unbounded_below}), and we have
    \begin{equation*}
        \mathcal{Y}_{u_{p+}}(\upsilon,\gamma_i)<\alpha(\upsilon)
        \quad\text{ and }\quad
        \mathcal{Y}_{u_{p-}}(\upsilon,\zeta_i)<\alpha(\upsilon),
    \end{equation*}
    for every $\upsilon\in[\upsilon_{\min},\upsilon_{\max}]$ and some $i\in\N_0$. Similar as before, from \eqref{eq:f1_monotonic_convergent} and \eqref{eq:f2_monotonic_convergent} we have that
    \begin{align*}
        f_1\left( \upsilon,\mathcal{Y}_{u_{p+}}\left(\upsilon,\gamma_i\right) \right) 
                - f_1\left( \upsilon,\alpha(\upsilon) \right) &\geq 0,\\
        f_2\left( \upsilon,\mathcal{Y}_{u_{p-}}\left(\upsilon,\zeta_i\right) \right) 
                - f_2\left( \upsilon,\alpha(\upsilon) \right) &\leq 0,
    \end{align*}
    for every $\upsilon\in[\upsilon_{\min},\upsilon_{\max}]$. 
    Consequently, the right term of \eqref{eq:gamma_i+1-gamma_i+-alpha} is positive or zero, which is a contradiction since by the assumption the sequence is strictly decreasing  and $\gamma_{i+1}-\gamma_i<0$ for every $i\in\N_0$.
    
    Therefore, both sequences are bounded, and since by Lemma \ref{lemma:sequences_monotonic} they are monotonic, then they are convergent.
    \begin{figure} 
        \centering
        \begin{subfigure}{0.50\linewidth}
            \centering
            \includegraphics[width=1.0\linewidth]{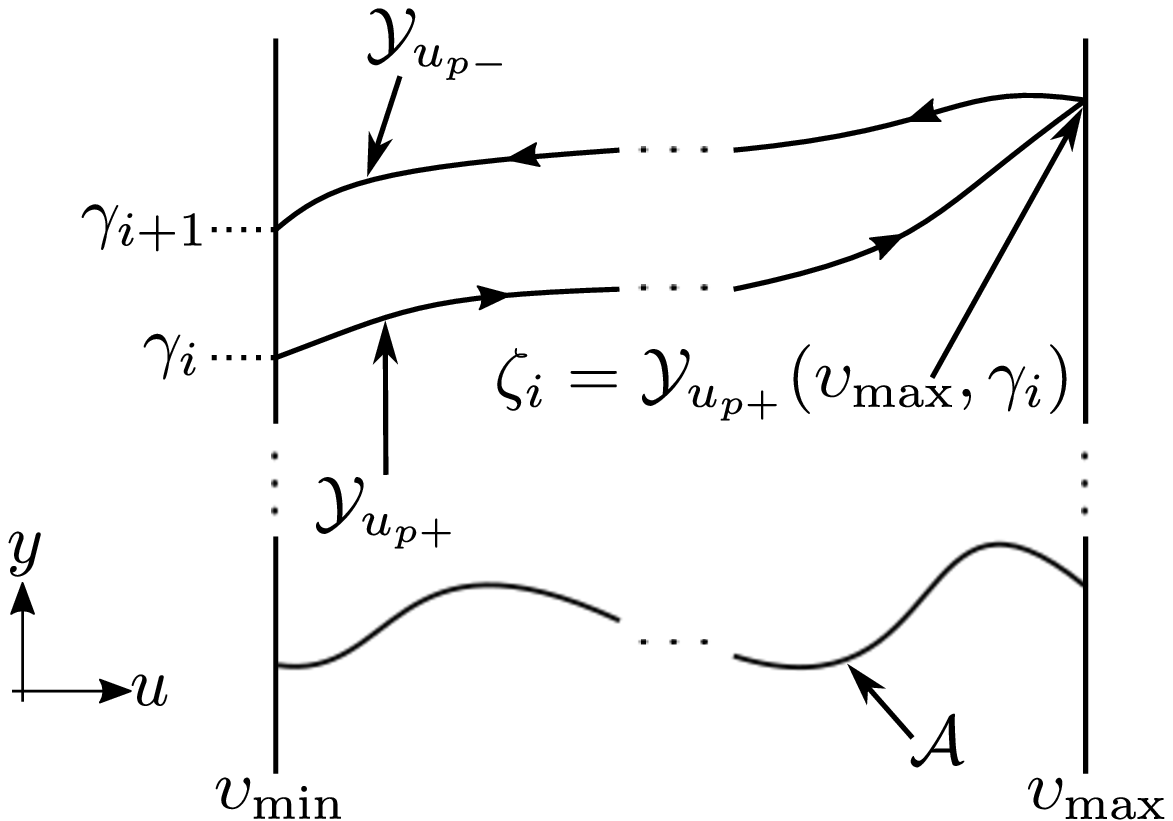}
            \caption{Strictly increasing $(\gamma_n)_{n\in\N_0}$. \label{fig:gamma_unbounded_above}}
        \end{subfigure}%
        \hspace*{1mm}%
        \begin{subfigure}{0.50\linewidth}
            \centering
            \includegraphics[width=1.0\linewidth]{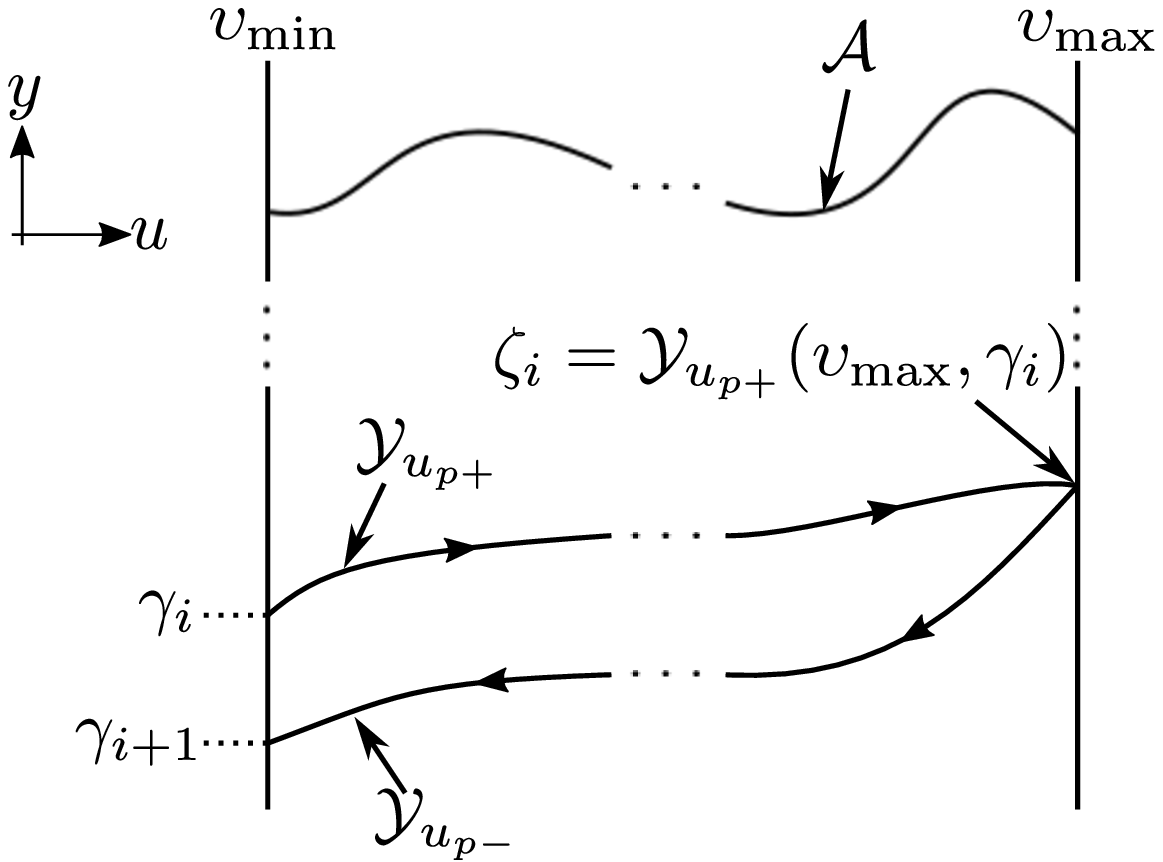}
            \caption{Strictly decreasing $(\gamma_n)_{n\in\N_0}$. \label{fig:gamma_unbounded_below}}
        \end{subfigure}
        \caption{Contradiction cases in proof of Proposition \ref{prop:sequences_bounded_convergent} for strictly monotonic unbounded sequence $(\gamma_n)_{n\in\N_0}$.}
    \end{figure}
\end{proof}
\vspace*{1mm}

\begin{proposition}\label{prop:sequences_convergent_unique}
    If the functions $f_1$ and $f_2$ in \eqref{eq:duhem_model} satisfy the strict version of inequalities \eqref{eq:f1_monotonic_convergent} and \eqref{eq:f2_monotonic_convergent} given by
    \begin{align}
        \left( f_1(\upsilon,\gamma_1) - f_1(\upsilon,\gamma_2) \right)\left( \gamma_1 - \gamma_2 \right) &< 0, \label{eq:f1_strictly_monotonic_convergent}\\
        \left( f_2(\upsilon,\gamma_1) - f_2(\upsilon,\gamma_2) \right)\left( \gamma_1 - \gamma_2 \right) &> 0, \label{eq:f2_strictly_monotonic_convergent}
    \end{align}
    for every $\gamma_1\neq\gamma_2$ and $\upsilon\in\R$,
    then there exist a unique pair $\gamma_*,\zeta_*\in\R$ such that for every $\gamma_0\in\R$ 
    the sequences generated by \eqref{eq:zeta_n} and \eqref{eq:gamma_n+1} 
    satisfy $(\gamma_n)_{n\in\N_0}\to\gamma_*$ and $(\zeta_n)_{n\in\N_0}\to\zeta_*$ where $\zeta_*=\mathcal{Y}_{u_{p+}}(\upsilon_{\max},\gamma_*)$.
\end{proposition}
\vspace*{1mm}
\begin{proof}
    We proceed by contradiction, assuming that the sequences $(\gamma_n)_{n\in\N_0}$ and $(\zeta_n)_{n\in\N_0}$ approach different values if we used different initial values.
    For this, assume there exists $\bar{\gamma}\neq\gamma_*$ and two initial values $\gamma_i\neq\gamma_j$ such that 
    $\displaystyle\lim_{i\to\infty}(\gamma_i-\gamma_*)=0$ and $\displaystyle\lim_{j\to\infty}(\gamma_j-\bar{\gamma})=0$. 
    We can subtract both limits and use \eqref{eq:gamma_i+1-gamma_i} to obtain
    \begin{equation}\label{eq:(gamma_i-gamma_*)-(gamma_j-bargamma)}
        \begin{aligned}
            0&=\displaystyle\lim_{i\to\infty}(\gamma_i-\gamma_*) 
                - \displaystyle\lim_{j\to\infty}(\gamma_j-\bar{\gamma}) \\
            &=\int_{\upsilon_{\min}}^{\upsilon_{\max}} \Big\{ 
                        f_1\left( \upsilon,\mathcal{Y}_{u_{p+}}\left(\upsilon,\gamma_*\right) \right) 
                      - f_2\left( \upsilon,\mathcal{Y}_{u_{p-}}\left(\upsilon,\zeta_*\right) \right)
                        \Big\}\ \dd\upsilon\\
            &\hphantom{=}-\int_{\upsilon_{\min}}^{\upsilon_{\max}} \Big\{ 
                        f_1\left( \upsilon,\mathcal{Y}_{u_{p+}}\left(\upsilon,\bar{\gamma}\right) \right) 
                      - f_2\left( \upsilon,\mathcal{Y}_{u_{p-}}\left(\upsilon,\bar{\zeta}\right) \right)
                        \Big\}\ \dd\upsilon\\
            &=\int_{\upsilon_{\min}}^{\upsilon_{\max}} \Big\{ 
                f_1\left( \upsilon,\mathcal{Y}_{u_{p+}}\left(\upsilon,\gamma_*\right) \right) 
              - f_1\left( \upsilon,\mathcal{Y}_{u_{p+}}\left(\upsilon,\bar{\gamma}\right) \right) 
                \Big\}\ \dd\upsilon \\
            &\hphantom{=}-\int_{\upsilon_{\min}}^{\upsilon_{\max}} \Big\{ 
                f_2\left( \upsilon,\mathcal{Y}_{u_{p-}}\left(\upsilon,\zeta_*\right) \right)
              - f_2\left( \upsilon,\mathcal{Y}_{u_{p-}}\left(\upsilon,\bar{\zeta}\right) \right)
                \Big\}\ \dd\upsilon
        \end{aligned}
    \end{equation}
    where $\zeta_*=\mathcal{Y}_{u_{p+}}(\upsilon_{\max},\gamma_*)$ and $\bar{\zeta}=\mathcal{Y}_{u_{p+}}(\upsilon_{\max},\bar{\gamma})$. Then, by \eqref{eq:f1_strictly_monotonic_convergent} and \eqref{eq:f2_strictly_monotonic_convergent} and Lemma \ref{lemma:non_intersecting_calYu}, the right term of the last expression is positive (resp. negative) when $\gamma_*<\bar{\gamma}$ and $\zeta*<\bar{\zeta}$ (resp. $\gamma_*>\bar{\gamma}$ and $\zeta*>\bar{\zeta}$), which is a contradiction.
\end{proof}
\vspace*{1mm}

Finally, to complement the previous two propositions, we also establish conditions such that the sequences $(\gamma_n)_{n\in\N_0}$ and $(\zeta_n)_{n\in\N_0}$ generated by \eqref{eq:zeta_n} and \eqref{eq:gamma_n+1} are divergent.
The conditions are presented in the form of a corollary since they follow immediately from \eqref{eq:gamma_i+1-gamma_i+-alpha} and the analysis in Proposition \ref{prop:sequences_bounded_convergent}.

\begin{corollary}\label{coro:sequences_divergent}
    If the functions $f_1$ and $f_2$ in the Duhem model \eqref{eq:duhem_model} satisfy the reversed inequalities to \eqref{eq:f1_strictly_monotonic_convergent} and \eqref{eq:f2_strictly_monotonic_convergent}, which are given by
    \begin{align}
        \left( f_1(\upsilon,\gamma_1) - f_1(\upsilon,\gamma_2) \right)\left( \gamma_1 - \gamma_2 \right) &> 0, \label{eq:f1_monotonic_divergent}\\
        \left( f_2(\upsilon,\gamma_1) - f_2(\upsilon,\gamma_2) \right)\left( \gamma_1 - \gamma_2 \right) &< 0, \label{eq:f2_monotonic_divergent}
    \end{align}
    for every $\gamma_1\neq\gamma_2$ and $\upsilon\in\R$,
    then for every $\gamma_0\in\R$ 
    the sequences $(\zeta_n)_{n\in\N_0}$ and $(\gamma_n)_{n\in\N_0}$ generated by \eqref{eq:zeta_n} and \eqref{eq:gamma_n+1} are divergent.
\end{corollary}

\subsection{Case Example: the Bouc-Wen  model}

We use now the propositions and corollary presented in this section to study a particular case of the Bouc-Wen hysteresis model \cite{Bouc1967,Wen1976,Ismail2009}. The Bouc-Wen model is commonly used to describe relations between displacement and restoring force as input and output in piezoactuated mechanical systems and it is defined by 
\begin{equation*}
    \dot{y}(t) = \alpha\dot{u}(t) 
                - \beta \left\lvert y(t) \right\rvert^n \dot{u}(t) 
                - \zeta y(t) \left\lvert y(t) \right\rvert^{n-1} \left\lvert \dot{u}(t) \right\rvert,
\end{equation*}
where $\alpha,\beta,\zeta\in\R$ are model parameters.
The equation above can be also written as a Duhem model of the form \eqref{eq:duhem_model} whose vector field functions $f_1$ and $f_2$ are defined by
\begin{align*}
    f_1(\upsilon,\gamma) &:=\alpha
                    - \beta \left\lvert \gamma \right\rvert^n 
                    - \zeta \gamma \left\lvert \gamma \right\rvert^{n-1}, \\ %\label{eq:f1_bouc-wen}\\
    f_2(\upsilon,\gamma) &:=\alpha
                    - \beta \left\lvert \gamma \right\rvert^n 
                    + \zeta \gamma \left\lvert \gamma \right\rvert^{n-1}. %\label{eq:f2_bouc-wen}
\end{align*}
Using these $f_1$ and $f_2$ into %the conditions 
\eqref{eq:f1_monotonic_convergent} and \eqref{eq:f2_monotonic_convergent} of Proposition \ref{prop:sequences_bounded_convergent} we have
\begin{align*}
     \left[
     \left( \beta + \zeta \sign\left( \gamma_1 \right) \right) 
        \left\lvert \gamma_1 \right\rvert^{n}
    -\left( \beta + \zeta\sign \left( \gamma_2 \right) \right)
        \left\lvert \gamma_2 \right\rvert^{n}
    \right] 
    \left( \gamma_1 - \gamma_2 \right) &\geq 0, \\
    \left[
     \left( \beta - \zeta \sign\left( \gamma_1 \right) \right) 
        \left\lvert \gamma_1 \right\rvert^{n}
    -\left( \beta - \zeta\sign \left( \gamma_2 \right) \right)
        \left\lvert \gamma_2 \right\rvert^{n}
    \right] 
    \left( \gamma_1 - \gamma_2 \right) &\leq 0.
\end{align*}
Assuming without loss of generality that $\gamma_1>\gamma_2$, we obtain
\begin{align}
     \left[
     \left( \beta + \zeta \sign\left( \gamma_1 \right) \right) 
        \left\lvert \gamma_1 \right\rvert^{n}
    -\left( \beta + \zeta\sign \left( \gamma_2 \right) \right)
        \left\lvert \gamma_2 \right\rvert^{n}
    \right] &\geq 0, \label{eq:boucwen_convergence_condition_1}\\
    \left[
     \left( \beta - \zeta \sign\left( \gamma_1 \right) \right) 
        \left\lvert \gamma_1 \right\rvert^{n}
    -\left( \beta - \zeta\sign \left( \gamma_2 \right) \right)
        \left\lvert \gamma_2 \right\rvert^{n}
    \right] &\leq 0. \label{eq:boucwen_convergence_condition_2}
\end{align}
Note now that when $\gamma_1>\gamma_2\geq0$ or $0\geq\gamma_1>\gamma_2$, we can reduce \eqref{eq:boucwen_convergence_condition_1} and \eqref{eq:boucwen_convergence_condition_2} to
\begin{align*}
    \left( \beta + \zeta \right) \left( \left\lvert \gamma_1 \right\rvert^n - \left\lvert \gamma_2 \right\rvert^n \right) &\geq 0, \\
    \left( \beta - \zeta \right) \left( \left\lvert \gamma_1 \right\rvert^n - \left\lvert \gamma_2 \right\rvert^n \right) &\leq 0,
\end{align*}
respectively, which are trivially satisfied when
\begin{align}
    \beta+\zeta &\geq 0, \label{eq:boucwen_zetabeta_convergence_condition_1}\\
    \beta-\zeta &\leq 0. \label{eq:boucwen_zetabeta_convergence_condition_2}
\end{align}
Moreover, when $\gamma_1>0>\gamma_2$ we have
\begin{align*}
    \beta \left( \left\lvert \gamma_1 \right\rvert^n - \left\lvert \gamma_2 \right\rvert^n \right) + 
    \zeta \left( \left\lvert \gamma_1 \right\rvert^n + \left\lvert \gamma_2 \right\rvert^n \right) &\geq 0 \\
    \beta \left( \left\lvert \gamma_1 \right\rvert^n - \left\lvert \gamma_2 \right\rvert^n \right) - 
    \zeta \left( \left\lvert \gamma_1 \right\rvert^n + \left\lvert \gamma_2 \right\rvert^n \right) &\leq 0
\end{align*}
which are also satisfied for \eqref{eq:boucwen_zetabeta_convergence_condition_1} and \eqref{eq:boucwen_zetabeta_convergence_condition_2}.
%Analyzing each one of the possible cases for the sign functions which are 
%
%\begin{equation*}
%    \begin{aligned}
%        \gamma_1>\gamma_2\geq0 \\
%        0>\gamma_1>\geq\gamma_2 \\
%        \gamma_1>\gamma_2\geq0
%    \end{aligned}
%\end{equation*}
%It can be checked that both conditions above are satisfied for every $\gamma_1>\gamma_2$ when
%\begin{equation}
%    \beta+\zeta \geq 0 \quad\text{and}\quad \beta-\zeta \leq 0, \label{eq:boucwen_zetabeta_convergence_conditions}
%\end{equation}

\begin{figure}[htb]
    \centering
    \includegraphics[width=0.7\columnwidth]{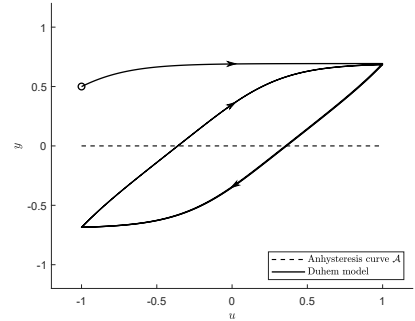}
    \caption{Hysteresis loop obtained from a Bouc-Wen hysteresis operator whose parameters $\alpha=1$, $\beta=2$, $\zeta=1$ satisfy the convergence conditions in \eqref{eq:boucwen_zetabeta_convergence_condition_1} and \eqref{eq:boucwen_zetabeta_convergence_condition_2} when a periodic input whose minimum and maximum are $\upsilon_{\min}=-1$ and $\upsilon_{\max}=1$ is applied. The initial point $(u(0),y(0))=(\upsilon_{\min},y_0)$ is marked by a circle. \label{fig:boucwen_convergent}}
\end{figure}

\begin{figure}[htb]
    \centering
    \includegraphics[width=0.7\columnwidth]{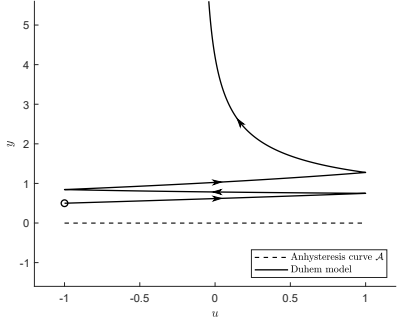}
    \caption{Divergent input-output phase plot obtained from a Bouc-Wen hysteresis operator whose parameters $\alpha=0.1$, $\beta=0.1$, $\zeta=-0.2$ satisfy the divergence conditions in
    \eqref{eq:boucwen_zetabeta_divergence_condition_1} and \eqref{eq:boucwen_zetabeta_divergence_condition_2}
    when a periodic input whose minimum and maximum are $\upsilon_{\min}=-1$ and $\upsilon_{\max}=1$ is applied. The initial point $(u(0),y(0))=(\upsilon_{\min},y_0)$ is marked by a circle. \label{fig:boucwen_divergent}}
\end{figure}

Therefore, the sequences defined by \eqref{eq:zeta_n} and \eqref{eq:gamma_n+1} are convergent for every initial value $\gamma_0\in\R$, or equivalently, the input-output phase plot of the Bouc-Wen model will converge to a periodic orbit from every initial point when conditions in \eqref{eq:boucwen_zetabeta_convergence_condition_1} and \eqref{eq:boucwen_zetabeta_convergence_condition_2} are satisfied. In fact, it can be checked that these conditions are equivalent to the ones presented in \cite[Table 1]{Ismail2009} corresponding to BIBO stable Bouc-Wen models of class I, III and V. As an illustrative example, the input-output phase plot of a Bouc-Wen model whose parameters satisfy the convergence conditions with $\alpha=1$, $\beta=1$ and $\zeta=2$ is shown in Fig. \ref{fig:boucwen_convergent}.

Conversely, based on Corollary \ref{coro:sequences_divergent}, when we have the reversed inequalities
\begin{align}
    \beta+\zeta &< 0, \label{eq:boucwen_zetabeta_divergence_condition_1}\\
    \beta-\zeta &> 0. \label{eq:boucwen_zetabeta_divergence_condition_2}
\end{align}
then the sequences defined by \eqref{eq:zeta_n} and \eqref{eq:gamma_n+1} will diverge which means that the input-output phase plot Bouc-Wen model will not exhibit a hysteresis loop. An example of this case is illustrated in Fig. \ref{fig:boucwen_divergent} with the divergent input-output phase plot of a Bouc-Wen model whose parameters are $\alpha=0.1$, $\beta=0.1$ and $\zeta=-0.2$.

%\todo[inline]{Add the mess of inequalities to show the equivalence with reference \cite[Table 1]{Ismail2009}?}

%The Dahl hysteresis model introduced for first time in \cite{Dahl1976} to describe solid friction force as a function of displacement is given by
%\begin{equation}\label{eq:dahl_model}
%    \dot{y}(t) = \sigma \left\lvert 1 - \frac{y(t)}{y_c} \right\rvert^{\eta} \sign \left( 1 - \frac{y(t)}{y_c}  \right) \dot{u}(t),
%\end{equation}
%where $u$ and $y$ correspond to the displacement and force applied and $\sigma$, $y_c$ and $\eta$ are parameters that determine the shape of the hysteresis loop obtained. The Dahl model given above can be written as a Duhem model of the form \eqref{eq:duhem_model} whose gradient functions $f_1$ and $f_2$ are defined by
%\begin{align}
%    f_1(\upsilon,\gamma) &:= \sigma \left\lvert 1 - \frac{\gamma}{y_c} \right\rvert^{\eta} \sign \left( 1 - \frac{\gamma}{y_c} \right), \label{eq:f1_dahl}\\
%    f_2(\upsilon,\gamma) &:= \sigma \left\lvert 1 + \frac{\gamma}{y_c} \right\rvert^{\eta} \sign \left( 1 + \frac{\gamma}{y_c} \right). \label{eq:f2_dahl}
%\end{align}
%We use this model as a particular case of study of the accommodation property studied in this section. 
%Using \eqref{eq:f1_dahl} and \eqref{eq:f2_dahl} within the conditions \eqref{eq:f1_monotonic_convergent} and \eqref{eq:f2_monotonic_convergent} of Proposition \ref{prop:sequences_bounded_convergent} we obtain
%\begin{equation}
%    
%\end{equation}

\section{THE DUHEM BUTTERFLY MODEL}
In this section we introduce a special class of Duhem operator which we call the Duhem butterfly operators. This operator is characterized by its capability in producing complex periodic hysteresis loops with self-intersections. In this class of operators both functions $f_1$ and $f_2$ in \eqref{eq:duhem_model} can assume positive and negative values as long as they satisfy the conditions \eqref{eq:f1_strictly_monotonic_convergent} and \eqref{eq:f2_strictly_monotonic_convergent}, respectively, to guarantee the existence of a unique periodic solution.

%{\color{red} Bayu is done up to here.}

We assume now that the implicit functions $\upsilon\mapsto\{\gamma\ |\ f_1(\upsilon,\gamma)=0\}$ and $\upsilon\mapsto\{\gamma\ |\ f_2(\upsilon,\gamma)=0\}$ admit explicit solutions
\begin{equation}\label{eq:level_set_c1_c2}
	\gamma=c_1(\upsilon)\qquad\text{and}\qquad\gamma=c_2(\upsilon),
\end{equation}
respectively, with $c_1,c_2\in AC(\R,\R)$ such that $f_1(\upsilon,c_1(\upsilon))=0$ and $f_2(\upsilon,c_2(\upsilon))=0$ for every $\upsilon\in\R$. In other words, the curves described by $c_1$ and $c_2$ are the zero level set of the functions $f_1$ and $f_2$, respectively. Note that by conditions \eqref{eq:f1_strictly_monotonic_convergent} and \eqref{eq:f2_strictly_monotonic_convergent}, each one of the curves $c_1$ and $c_2$ split the input-output plane $u-y$ into two regions such that
\begin{align*}
%    \begin{aligned}
        f_1(\upsilon,\gamma)<0 & \text{ whenever } \gamma>c_1(\upsilon);\\
        f_1(\upsilon,\gamma)>0 & \text{ whenever } \gamma<c_1(\upsilon);\\
        f_2(\upsilon,\gamma)>0 & \text{ whenever } \gamma>c_2(\upsilon); \text{and}\\
        f_2(\upsilon,\gamma)<0 & \text{ whenever } \gamma>c_2(\upsilon).
%    \end{aligned}
\end{align*}

In the following, we will prove that when the functions $f_1$ and $f_2$, and the zero-level set functions $c_1$ and $c_2$ satisfy some mild assumptions, %it is possible to find 
there is a periodic hysteresis loop with a self-intersection which gives the existence of a %and therefore produce 
butterfly hysteresis loop. 
Prior to this, %Throughout this analysis, we 
we need to introduce the following notations. Let $u_+,u_-\in AC(\R_+,\R)$ be inputs which are monotonically increasing and decreasing, respectively, and radially unbounded, i.e. $u_+(t)\to\infty$ and $u_-(t)\to -\infty$ as $t\to\infty$, respectively. 
Similar to \eqref{eq:integral_Yu+} and \eqref{eq:integral_Yu-}, we define the solutions of the Duhem model \eqref{eq:duhem_model} parameterized by the instantaneous value of the inputs $u_+$ and $u_-$ by $\mathcal{Y}_{u_+}$ and $\mathcal{Y}_{u_-}$, respectively. The next lemma shows that %where we prove, 
under mild assumptions on the functions $c_1$ and $c_2$, the positive invariance of the region below the curves $c_1$ and $c_2$ with respect to the solutions of $\mathcal{Y}_{u_+}$ and $\mathcal{Y}_{u_-}$, respectively. %(in the sense of increasing time) of the solutions $\mathcal{Y}_{u_+}$ and $\mathcal{Y}_{u_-}$ with respect to the region below the curves $c_1$ and $c_2$, respectively.\\

\begin{lemma}\label{lemma:sol_positive_invariance}
	If $0\leq\frac{dc_1(\upsilon)}{d\upsilon}\leq L_1$ for all $\upsilon\geq u_+(0)$ % and let $\upsilon_0=u_+(0)$. 
	then for all $\gamma_0\leq c_1(u_+(0))$,  $\mathcal{Y}_{u_+}(\upsilon,\gamma_0)\leq c_1(\upsilon)$ for all $\upsilon\geq u_+(0)$.
	
%	If $\gamma_0\leq c_1(\upsilon_0)$ then $\mathcal{Y}_{u_+}(\upsilon,\gamma_0)\leq c_1(\upsilon)$ for every $\upsilon\geq \upsilon_0$. \\ \vspace*{0.1mm}
    Analogously, if $-L_2\leq\frac{dc_2(\upsilon)}{d\upsilon}\leq 0$ for all $\upsilon\leq u_-(0)$ then for all $\gamma_0\leq c_2(u_-(0))$, %and let $\upsilon_0=u_-(0)$ and $\gamma_0\leq c_2(\upsilon_0)$ then 
    $\mathcal{Y}_{u_-}(\upsilon,\gamma_0)\leq c_2(\upsilon)$ for all $\upsilon\leq \upsilon_0$.
\end{lemma}
\begin{proof}
    \begin{figure}
        \centering
        \begin{subfigure}{0.45\linewidth}
            \centering
            \includegraphics[width=1.0\linewidth]{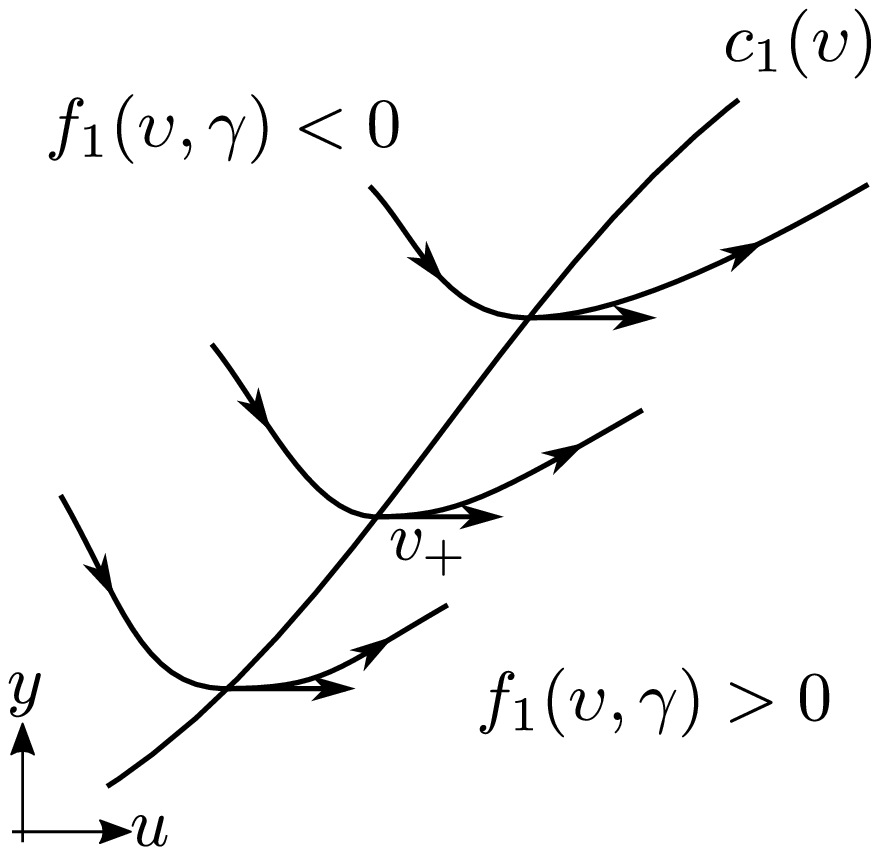}
            \caption{With a increasing input $u_+$. \label{fig:invariant_+}}
        \end{subfigure}%
        \hspace*{5mm}%
        \begin{subfigure}{0.45\linewidth}
            \centering
            \includegraphics[width=1.0\linewidth]{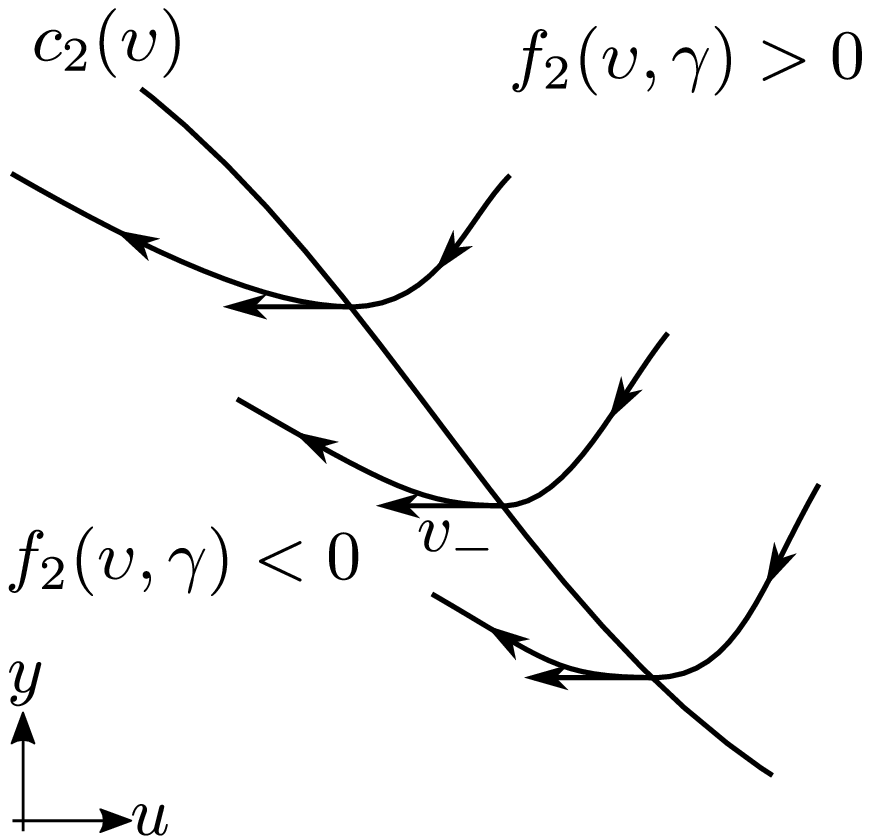}
            \caption{With a decreasing input $u_-$. \label{fig:invariant_-}}
        \end{subfigure}
        \caption{Invariance of the solutions for an increasing and decreasing input respect to the parameterizations $c_1(\upsilon)$ and $c_2(\upsilon)$ of the level sets $f_1(\upsilon,\gamma)=0$ and $f_2(\upsilon,\gamma)=0$, respectively.}
    \end{figure}
    We prove now the first claim of the lemma. Let us define the domain under the curve $c_1$ as follows %To prove the first part of this lemma, consider the set
    \begin{equation}\label{eq:C_1+}
        \mathcal{C}_{1+}:=\left\{ (\upsilon,\gamma)\in\R^2\ |\ \gamma\leq c_1(\upsilon) \right\}.
    \end{equation}
%    which consists of all the points below de curve paremeterized by $\gamma=c_1(\upsilon)$. We will prove that 
It can be checked that $\mathcal{C}_{1+}$ is positively invariant with respect to the solutions of Duhem model \eqref{eq:duhem_model} with monotonically increasing input $u_+$ and with initial conditions in $\mathcal{C}_{1+}$. Indeed, for every point $x\in\mathcal{C}_{1+}$ we can construct the tangent cone to this set as defined in \cite[Def. 3.1]{Blanchini1999}, which is given by
    \begin{equation*}
        \mathcal{T}_{\mathcal{C}_{1+}}(x) = \left\{ z\in\R^2\colon \liminf_{h\to 0} \frac{ \text{dist}\left(x + hz, \mathcal{C}_{1+} \right) }{ h } = 0 \right\},
    \end{equation*}
    where we take $\text{dist}(\cdot)$ to be the Euclidian distance from $x$ to the closest point $y\in\mathcal{C}_{1+}$.
    Let $\nu_{+}(\upsilon_0)\in\R^2$ be the tangent vector to the solution $\mathcal{Y}_{u_+}$ which is given by
    \begin{equation*}
        \begin{aligned}
        \nu_{+}(\upsilon_0) &= \left(1,\left.\frac{\dd}{\dd \upsilon}\right|_{\upsilon=\upsilon_0} \mathcal{Y}_{u_+}\big( \upsilon, c_1(\upsilon_0) \big) \right) \\
        &= \Big(1,f_1\big(\upsilon_0,c_1(\upsilon_0)\big)\Big).
        \end{aligned}
    \end{equation*}
    Now, we show that $\nu_{+}(\upsilon_0)\in\mathcal{T}_{\mathcal{C}_{1+}}(x_{1})$ with $x_{1}=(\upsilon_0,c_1(\upsilon_0))\in\mathcal{C}_{1+}$ for every $\upsilon_0\in\R$ so that the solutions of $\mathcal Y_{u_+}$ do not escape $\mathcal C_{1+}$ on the boundary (see Fig. \ref{fig:invariant_+}). In other words, we show that the tangent vector to the solution $\mathcal{Y}_{u_+}$ belongs to the tangent cone to the set $\mathcal{C}_{1+}$ at every point of the boundary. For this let us consider a point $w=(\upsilon_0+h,c_1(\upsilon_0))$ and note that since $c_1(\upsilon)$ is monotonically increasing we have $w\in \mathcal{C}_{1+}$ for every $h>0$. Then we can check
    \begin{equation*}
        \begin{aligned}
             &\liminf_{h\to 0^+} \frac{ \lVert \big(x_{1} + h\nu_{+}(\upsilon_0)\big) - w \rVert}{ h }\\
             &\quad=\liminf_{h\to 0^+} \frac{   h \left|f_1\big(\upsilon_0,c_1(\upsilon_0)\big)\right|  
             }{ h }
             =\left|f_1\big(\upsilon_0,c_1(\upsilon_0)\big)\right| = 0,
        \end{aligned}
    \end{equation*}
    which proves that $\nu_{+}(\upsilon_0)\in\mathcal{T}_{\mathcal{C}_{2-}}(x_{1})$. Consequently, following from Nagumo theorem \cite[Th. 3.1]{Blanchini1999} the set $\mathcal{C}_{1+}$ is positively invariant and $\mathcal{Y}_{u_+}(\upsilon,\gamma_0)\leq c_1(\upsilon)$ for every $\upsilon\geq \upsilon_0$.\\

    \noindent For proving the second claim of the lemma, %analogous part of this lemma, 
    we consider the set
    \begin{equation}\label{eq:C_2-}
        \mathcal{C}_{2-}:=\left\{ (\upsilon,\gamma)\in\R^2\ |\ \gamma\leq c_2(\upsilon) \right\},
    \end{equation}
    which consists of all the points below the curve parameterized by $\gamma=c_2(\upsilon)$. 
    We let $\nu_{-}(\upsilon_0)\in\R^2$ be the tangent vector to the solution $\mathcal{Y}_{u_-}$ given by
    \begin{equation*}
        \begin{aligned}
            \nu_{-}(\upsilon_0) &= \left(-1,\left.\frac{\dd}{\dd \upsilon}\right|_{u=\upsilon_0} \mathcal{Y}_{u_-}\big( \upsilon, c_2(\upsilon_0) \big) \right)\\
            &=\Big(-1,f_2\big(\upsilon_0,c_2(\upsilon_0)\big)\Big).
        \end{aligned}
    \end{equation*}
    In this case, we show that the tangent vector $\nu_{-}(\upsilon_0)$ to the solution $\mathcal{Y}_{u_-}$ belongs to the tangent cone to the set $\mathcal{C}_{2-}$ at every point of the boundary (see Fig. \ref{fig:invariant_-}). We consider in this case a point $w=(\upsilon_0-h,c_2(\upsilon_0))$ and note that since $c_2(\upsilon)$ is monotonically decreasing we have $w\in \mathcal{C}_{2-}$ for every $h>0$. Then we can check analogously that
    \begin{equation*}
    \begin{aligned}
    &\liminf_{h\to 0^+} \frac{ \lVert \big(x_{2} + h\nu_{-}(\upsilon_0)\big) - w \rVert}{ h }\\
    &\quad=\liminf_{h\to 0^+} \frac{ h \lvert f_2\big(\upsilon_0,c_2(\upsilon_0) \rvert }{ h }
    = \lvert f_2\big(\upsilon_0,c_2(\upsilon_0)\big) \rvert = 0,
    \end{aligned}
    \end{equation*}
    which proves that $\nu_{-}(\upsilon_0)\in\mathcal{T}_{\mathcal{C}_{2-}}(x_{2})$ and following again from Nagumo theorem the set $\mathcal{C}_{2-}$ is positively invariant and $\mathcal{Y}_{u_-}(\upsilon,\gamma_0)\leq c_2(\upsilon)$ for every $\upsilon\leq \upsilon_0$.
\end{proof}\vspace*{1mm}

We remark that Lemma \ref{lemma:sol_positive_invariance} proves invariance of the solutions only for the case when the slopes of the level set functions $c_1$ and $c_2$ in \eqref{eq:level_set_c1_c2} are positive and negative, respectively. Nevertheless, it is also possible to prove invariance for the opposite case corresponding to the level set functions $c_1$ and $c_2$ having negative and positive slopes, respectively. In this opposite case, the invariant set for $\mathcal{Y}_{u_+}$ and $\mathcal{Y}_{u_-}$ correspond to the closure of the complement of $\mathcal C_{1+}$ in \eqref{eq:C_1+} and $\mathcal C_{2-}$ in \eqref{eq:C_2-}, respectively. % upper partition and lower partitions of the $u-y$ plane, respectively.

In the next lemma we prove that under mild assumptions regarding the monotonicity in the first argument of the functions $f_1$ and $f_2$, the extended solutions $\mathcal{Y}_{u_+}$ and $\mathcal{Y}_{u_-}$ in the reverse direction (when the input signal $u_+$ and $u_-$ as defined before Lemma \ref{lemma:sol_positive_invariance} are extended from $\rline_+$ to the whole real $\rline$) intersect with the zero level set curve $c_2$ and $c_1$, respectively.

\begin{lemma}\label{lemma:sol_extension_outbounded}
	Assume that the hypotheses in Lemma \ref{lemma:sol_positive_invariance} hold. Suppose that 
	$f_1$ satisfy
	\begin{align}
		\left( f_1(\upsilon_1,\gamma) - f_1(\upsilon_2,\gamma) \right)\left( \upsilon_1 - \upsilon_2 \right) &< 0, \label{eq:f1_upsilon_strictly_monotonic}
	\end{align}
	for every $\upsilon_1,\upsilon_2,\gamma\in\R$ and 
    let $\upsilon_a,\gamma_a\in\R$ be such that $\gamma_a=c_1(\upsilon_a)<c_2(\upsilon_a)$. 
    Then there exists $\upsilon_b<\upsilon_a$ such that $\mathcal{Y}_{u_+}(\upsilon_b,\gamma_a)=c_2(\upsilon_b)$.
    
    Analogously, suppose that $f_2$ satisfy
    \begin{align}
    	\left( f_2(\upsilon_1,\gamma) - f_2(\upsilon_2,\gamma) \right)\left( \upsilon_1 - \upsilon_2 \right) &> 0, \label{eq:f2_upsilon_strictly_monotonic}
    \end{align}
	for every $\upsilon_1,\upsilon_2,\gamma\in\R$ and 
    let $\upsilon_a,\gamma_a\in\R$ be such that $\gamma_a=c_1(\upsilon_a)>c_2(\upsilon_a)$.
    Then there exists  $\upsilon_b>\upsilon_a$ such that $\mathcal{Y}_{u_-}(\upsilon_b,\gamma_a)=c_1(\upsilon_b)$.
\end{lemma}
\begin{proof}
    \begin{figure}
        \centering
        \begin{subfigure}{0.48\linewidth}
            \centering
            \includegraphics[width=1.0\linewidth]{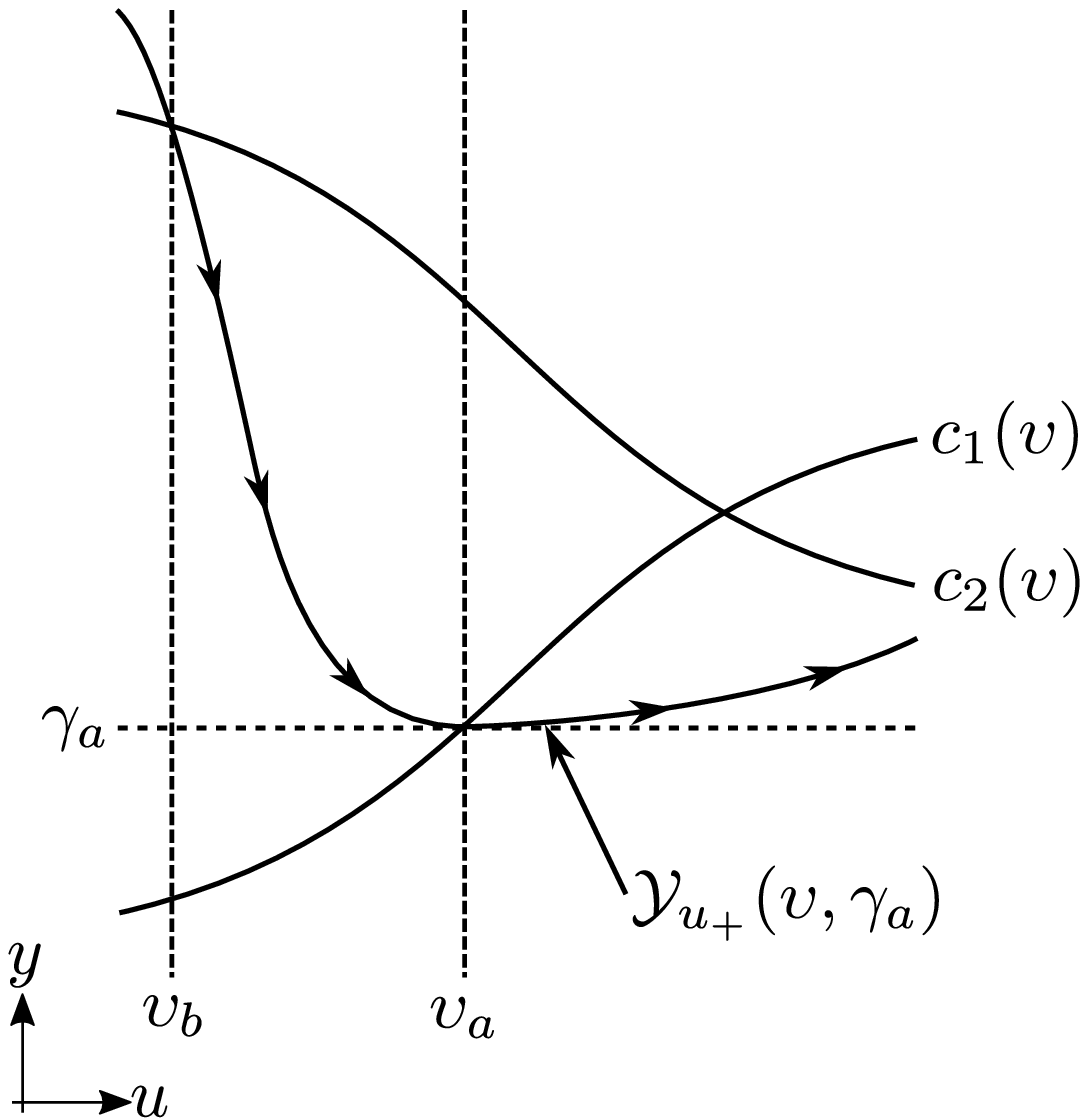}
            \caption{Intersection of the level set curve $\gamma=c_1(\upsilon)$ and the solution $\mathcal{Y}_{u_-}(\upsilon,\gamma_a)$. \label{fig:outbounded_y+}}
        \end{subfigure}%
        \hspace*{2mm}%
        \begin{subfigure}{0.48\linewidth}
            \centering
            \includegraphics[width=1.0\linewidth]{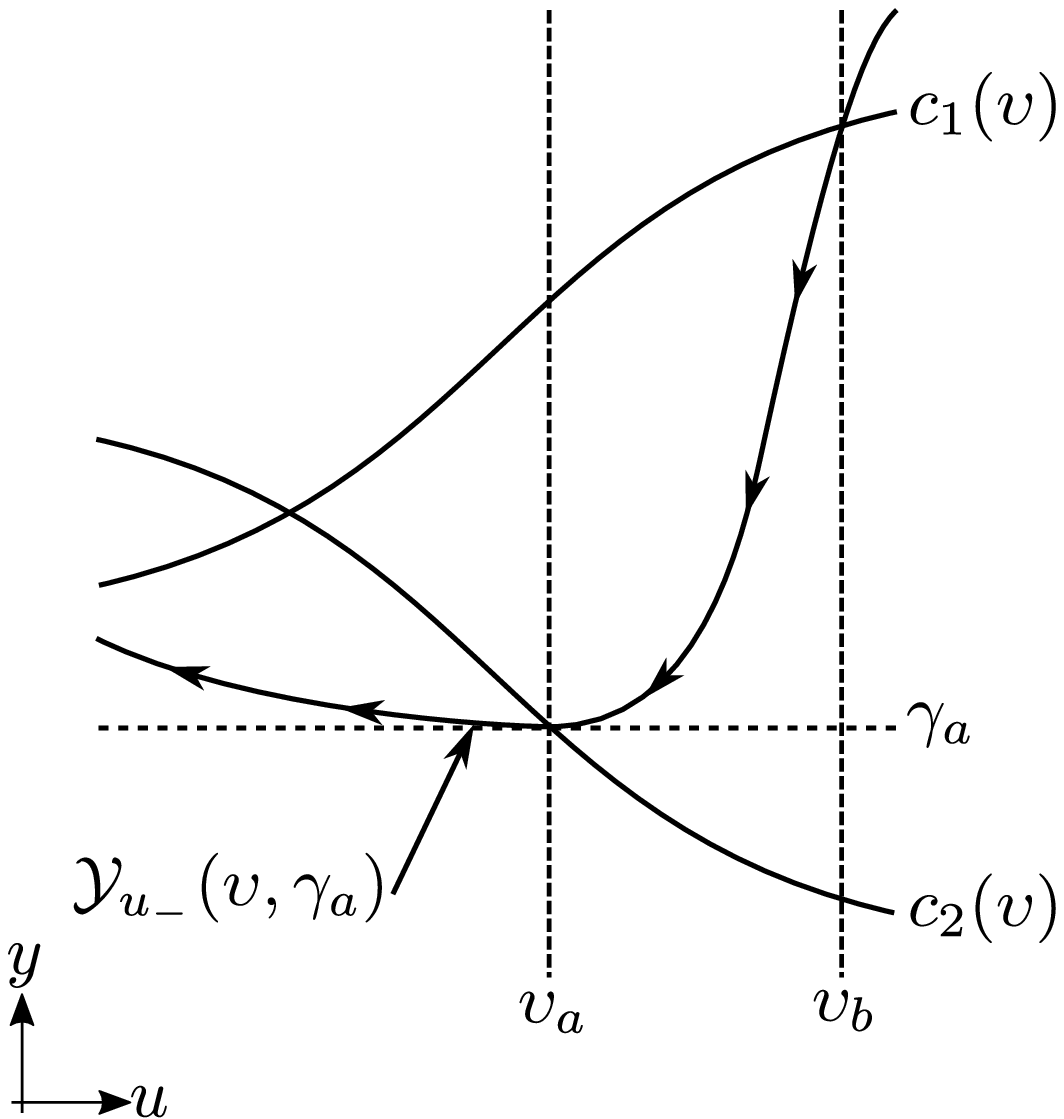}
            \caption{Intersection of the level set curve $\gamma=c_2(\upsilon)$ and the solution $\mathcal{Y}_{u_+}(\upsilon,\gamma_a)$. \label{fig:outbounded_y-}}
        \end{subfigure}
        \caption{Intersection of the solutions for an increasing and decreasing input with the parameterizations $c_1(\upsilon)$ and $c_2(\upsilon)$ of the level sets $f_1(\upsilon,\gamma)=0$ and $f_2(\upsilon,\gamma)=0$, respectively.}
    \end{figure}   
    Let us firstly prove the existence of a point $\upsilon_b$ where the curve $\mathcal Y_{u+}(\cdot,\gamma_a)$ intersects with $c_2$ at $\upsilon_b$. By extending $u_+$ from $\rline_+$ to $\rline$ while still satisfying the monotonicity and radial unbounded assumption of $u_+$ (e.g.,
    $\displaystyle{\lim_{t\to -\infty}}u_+(t)=-\infty$ 
    and $\displaystyle{\lim_{t\to + \infty}}u_+(t)=\infty$), 
    the solution $\mathcal{Y}_{u_+}(\upsilon,\gamma_a)$ can be extended in the negative direction (i.e. $\upsilon<\upsilon_a$) and the equation
    \begin{equation*}
        \begin{aligned}
            \mathcal{Y}_{u_+}(\upsilon,\gamma_a) &=\hphantom{-}\int_{\upsilon_a}^{\upsilon} f_1\left(\upsilon, \mathcal{Y}_{u_+}(\upsilon,\gamma_a) \right)\ \dd\upsilon + \gamma_a\\
            &=-\int_{\upsilon}^{\upsilon_a} f_1\left(\upsilon, \mathcal{Y}_{u_+}(\upsilon,\gamma_a) \right)\ \dd\upsilon + \gamma_a,
        \end{aligned}
    \end{equation*}
    is still valid (see Fig. \ref{fig:outbounded_y+}). 
    Moreover, since $f_1(\upsilon,\gamma)<0$ whenever $\gamma> c_1(\upsilon)$ we have that
    \begin{equation*}
        \begin{aligned}
            \mathcal{Y}_{u_+}(\upsilon,\gamma_a) &= \left| \int_{\upsilon}^{\upsilon_a} f_1\left(\upsilon, \mathcal{Y}_{u_+}(\upsilon,\gamma_a) \right)\ \dd\upsilon \right| + \gamma_a,
        \end{aligned}
    \end{equation*}
    for every $\upsilon<\upsilon_a$, which means that the extension of the solution $\mathcal{Y}_{u_+}(\upsilon,\gamma_a)$ in the negative direction remains above the curve parameterized by $\gamma=c_1(\upsilon)$.
    By the assumption  \eqref{eq:f1_upsilon_strictly_monotonic} and using the bound $L_2$ of $\frac{dc_2(\upsilon)}{d\upsilon}$ as in the hypotheses of Lemma \ref{lemma:sol_positive_invariance}, we have that there exists $\upsilon_{L_2}\leq\upsilon_a$ such that for every $\upsilon<\upsilon_{L_2}$ we have
    \begin{equation*}
        f_1\left(\upsilon, \mathcal{Y}_{u_+}(\upsilon,\gamma_a) \right) 
        < f_1\left(\upsilon_{L_2}, \mathcal{Y}_{u_+}(\upsilon_{L_2},\gamma_a) \right) 
        = -L_2.
    \end{equation*}
    Since we have that 
    \begin{equation*}
        \begin{aligned}
            c_2(\upsilon) &= \int_{\upsilon_a}^{\upsilon} \frac{\dd c_2}{\dd\upsilon}\dd\upsilon + c_2(\upsilon_a)\\
            &= -\int_{\upsilon}^{\upsilon_a} \frac{\dd c_2}{\dd\upsilon}\dd\upsilon + c_2(\upsilon_a) 
            \leq L_2(\upsilon_a-\upsilon) + c_2(\upsilon_a),
        \end{aligned}
    \end{equation*}
    the solution $\mathcal{Y}_{u_+}(\upsilon,\gamma_a)$ and the curve parameterized by $\gamma=c_2(\upsilon)$ intersect each other at some $\upsilon_b<\upsilon_{L_2}$. Indeed, this can be observed from the fact that
    \begin{equation*}
        \begin{aligned}
            &\mathcal{Y}_{u_+}(\upsilon,\gamma_a)-c_2(\upsilon)=\gamma_a-c_2(\upsilon_a)\\
            &\qquad\qquad+\int_{\upsilon_a}^{\upsilon} \left\{ f_1\left(\upsilon, \mathcal{Y}_{u_+}(\upsilon,\gamma_a) \right) - \frac{\dd c_2}{\dd\upsilon} \right\}\ \dd\upsilon \\
            &\qquad=
            	\underbrace{\int_{\upsilon_a}^{\upsilon_{L_2}} \left\{ f_1\left(\upsilon, \mathcal{Y}_{u_+}(\upsilon,\gamma_a) \right) - \frac{\dd c_2}{\dd\upsilon} \right\}\ \dd\upsilon}_{<0} \\
            &\qquad\qquad+
                \underbrace{\int_{\upsilon_{L_2}}^{\upsilon} \left\{ f_1\left(\upsilon, \mathcal{Y}_{u_+}(\upsilon,\gamma_a) \right) - \frac{\dd c_2}{\dd\upsilon}, \right\}\ \dd\upsilon}_{\geq 0}
                \\
        \end{aligned}
    \end{equation*}
    where the last term grows radially unbounded for $\upsilon<\upsilon_{L_2}$.\\
    
    We can prove analogously the second claim of the lemma as illustrated in Fig. \ref{fig:outbounded_y-}. Similar as before, the solution $\mathcal{Y}_{u_-}(\upsilon,\gamma_a)$ can be extended in the positive direction (i.e. $\upsilon>\upsilon_a$) when $u_-$ is extended from $\rline_+$ to $\rline$  satisfying the monotonicity and radial unbounded assumption of $u_-$. In this case, 
    \begin{equation*}
            \mathcal{Y}_{u_-}(\upsilon,\gamma_a) 
            = \int_{\upsilon_a}^{\upsilon} f_2\left(\upsilon, \mathcal{Y}_{u_-}(\upsilon,\gamma_a) \right)\ \dd\upsilon + \gamma_a
    \end{equation*}
    and since $f_2(\upsilon,\gamma)>0$ whenever $\gamma>c_2(\upsilon)$,  we have that
    \begin{equation*}
        \begin{aligned}
            \mathcal{Y}_{u_-}(\upsilon,\gamma_a) &= \left| \int_{\upsilon}^{\upsilon_a} f_2\left(\upsilon, \mathcal{Y}_{u_-}(\upsilon,\gamma_a) \right)\ \dd\upsilon \right| + \gamma_a,
        \end{aligned}
    \end{equation*}
    for every $\upsilon>\upsilon_a$, which means that the extension of the solution $\mathcal{Y}_{u_-}(\upsilon,\gamma_a)$ in the positive direction remains above the curve parameterized by $\gamma=c_2(\upsilon)$.
    In this case, using \eqref{eq:f2_upsilon_strictly_monotonic} and the bound $L_1$ of $\frac{dc_1(\upsilon)}{d\upsilon}$, we have that there exists $\upsilon_{L_1}\geq\upsilon_a$ such that for every $\upsilon>\upsilon_{L_1}$ we have
    \begin{equation*}
        f_2\left(\upsilon, \mathcal{Y}_{u_-}(\upsilon,\gamma_a) \right) 
        > f_2\left(\upsilon_{L_1}, \mathcal{Y}_{u_-}(\upsilon_{L_1},\gamma_a) \right) 
        = L_1.
    \end{equation*}
    Since we have that  
    \begin{equation*}
        \begin{aligned}
            c_1(\upsilon) = \int_{\upsilon_a}^{\upsilon} \frac{\dd c_1}{\dd\upsilon}\dd\upsilon + c_1(\upsilon_a)
            \leq L_1(\upsilon-\upsilon_a) + c_1(\upsilon_a),
        \end{aligned}
    \end{equation*}
    the solution $\mathcal{Y}_{u_-}(\upsilon,\gamma_a)$ and the curve parameterized by $\gamma=c_1(\upsilon)$ intersect each other at some $\upsilon_b>\upsilon_{L_1}$. It follows from the fact that % which follows from noting that in
    \begin{equation*}
        \begin{aligned}
            &\mathcal{Y}_{u_-}(\upsilon,\gamma_a)-c_1(\upsilon)=\gamma_a-c_1(\upsilon_a)\\
            &\qquad\qquad+\int_{\upsilon_a}^{\upsilon} \left\{ f_2\left(\upsilon, \mathcal{Y}_{u_-}(\upsilon,\gamma_a) \right) - \frac{\dd c_1}{\dd\upsilon} \right\}\ \dd\upsilon \\
            &\qquad=\underbrace{\int_{\upsilon_a}^{\upsilon_{L_1}} \left\{ f_2\left(\upsilon, \mathcal{Y}_{u_-}(\upsilon,\gamma_a) \right) - \frac{\dd c_1}{\dd\upsilon} \right\}\ \dd\upsilon}_{<0} \\
            &\qquad\qquad+
                \underbrace{\int_{\upsilon_{L_1}}^{\upsilon} \left\{ f_2\left(\upsilon, \mathcal{Y}_{u_-}(\upsilon,\gamma_a) \right) - \frac{\dd c_1}{\dd\upsilon}, \right\}\ \dd\upsilon}_{\geq 0}
                \\
        \end{aligned}
    \end{equation*}
    where the last term grows radially unbounded for $\upsilon>\upsilon_{L_1}$.
\end{proof}

As with Lemma \ref{lemma:sol_positive_invariance}, we also remark that Lemma \ref{lemma:sol_extension_outbounded} proves that the extension of the solutions in the negative direction of their corresponding input intersect with the zero level set functions $c_1$ and $c_2$ only for the case when their slopes are positive and negative, respectively. However, vis-a-vis arguments can prove the opposite case when the extended solutions in the negative direction of their corresponding input intersect with the level set functions $c_1$ and $c_2$ have negative and positive slopes, respectively.

In the following proposition we present the main result of this section, where we prove constructively the existence of outputs $\mathcal{Y}_{u_+}$ and $\mathcal{Y}_{u_-}$ with intersections.\\

\begin{proposition}\label{prop:duhem_butterfly}
	Assume that the hypotheses in Lemma \ref{lemma:sol_extension_outbounded} are satisfied (which include those in Lemmas \ref{lemma:sol_positive_invariance}). Let $\upsilon_f\in\R$ be such that $c_1(\upsilon_f)=c_2(\upsilon_f)$. % with $\upsilon_{a_+}<\upsilon_f$.
    Then for every $\upsilon_{a_+}<\upsilon_f$ there exist $\upsilon_{\min},\upsilon_x,\upsilon_{a_-},\upsilon_{\max}\in\R$ such that $\upsilon_{\min}<\upsilon_{a_+}<\upsilon_x<\upsilon_{a_-}<\upsilon_{\max}$ and
	\begin{equation*}
        \begin{aligned}
      		\mathcal{Y}_{u_+}(\upsilon_{\min},c_1(\upsilon_{a_+}))
      		&=
      		\mathcal{Y}_{u_-}(\upsilon_{\min},c_2(\upsilon_{a_-}))\\
            \mathcal{Y}_{u_+}(\upsilon_x,c_1(\upsilon_{a_+}))
            &= 
            \mathcal{Y}_{u_-}(\upsilon_x,c_2(\upsilon_{a_-}))\\
            \mathcal{Y}_{u_+}(\upsilon_{\max},c_1(\upsilon_{a_+}))
            &=
            \mathcal{Y}_{u_-}(\upsilon_{\max},c_2(\upsilon_{a_-})).\\
        \end{aligned}
	\end{equation*}
    In other words, the solutions $\mathcal{Y}_{u_+}(\cdot,c_1(\upsilon_{a_+}))$ and $	\mathcal{Y}_{u_-}(\cdot,c_2(\upsilon_{a_-}))$ which intersect $c_1$ and $c_2$ at $\upsilon_{a_+}$ and $\upsilon_{a_-}$, respectively, intersect also each other at $\upsilon_x$, $\upsilon_{\min}$ and $\upsilon_{\max}$.
\end{proposition}
\begin{figure}
	\centering
	\includegraphics[width=1.0\linewidth]{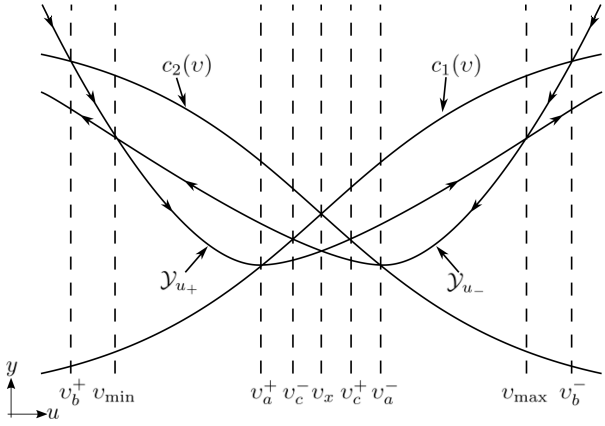}
	\caption{Construction of butterfly loop from the intersections of two solutions $\mathcal{Y}_{u_+}$ and $\mathcal{Y}_{u_-}$.\label{fig:butterfly_y+_y-}}
\end{figure}
\begin{proof}
    For a better understanding of the constructive proof of this proposition we refer the reader to Fig. \ref{fig:butterfly_y+_y-}.\\
    Consider the solution $\mathcal{Y}_{u_+}(\upsilon,c_1(\upsilon_{a_+}))$. By Lemma \ref{lemma:sol_extension_outbounded} there exists $\upsilon_{b_+}<\upsilon_{a_+}$ where this solution intersects the curve $c_2$ i.e.
    \begin{equation*}    
        \mathcal{Y}_{u_+}(\upsilon_{b_+},c_1(\upsilon_{a_+})) = c_2(\upsilon_{b_+}).
    \end{equation*}
    Additionally, by Lemma \ref{lemma:sol_positive_invariance} we have that the solution $\mathcal{Y}_{u_+}$ remains below the curve $c_1$ for every $\upsilon>\upsilon_{a_+}$ but always increasing as $\upsilon$ increases since $f_1(\upsilon,\gamma)>0$ when $\gamma<c_1(\upsilon)$. Therefore, since  $\frac{dc_2}{d\upsilon}\leq0$, then the solution $\mathcal{Y}_{u_+}$ must also  intersect the curve $c_2$ at some $\upsilon_{c_+}>\upsilon_{a_+}$ i.e.
    \begin{equation*}    
        \mathcal{Y}_{u_+}(\upsilon_{c_+},c_1(\upsilon_{a_+})) = c_2(\upsilon_{c_+}).
    \end{equation*}
    
    \noindent Let us define now $\upsilon_{a_-} = \upsilon_{c_+} + \varepsilon$ with $\varepsilon>0$ being arbitrarily small and consider the solution $\mathcal{Y}_{u_-}(\upsilon,c_2(\upsilon_{a_-}))$. 
    As in the previous case, by Lemma \ref{lemma:sol_extension_outbounded} there exists $\upsilon_{b_-}>\upsilon_{a_-}$ where this solution intersects the curve $c_1$ i.e. 
    \begin{equation*}    
        \mathcal{Y}_{u_-}(\upsilon_{b_-},c_2(\upsilon_{a_-})) = c_1(\upsilon_{b_-}).
    \end{equation*}
    We can also note that by Lemma \ref{lemma:sol_positive_invariance} the solution $\mathcal{Y}_{u_-}$ remains below the curve $c_2$ but always increasing as $\upsilon$ decreases given that $f_2(\upsilon,\gamma)>0$ for every $\upsilon<\upsilon_{a_-}$.
    Consequently, since  $\frac{dc_1}{d\upsilon}\geq0$, then the solution $\mathcal{Y}_{u_-}$ must also intersect the curve $c_1$ at some $\upsilon_{c_-}<\upsilon_{a_-}$ i.e. 
    \begin{equation*}
        \mathcal{Y}_{u_-}(\upsilon_{c_-},c_2(\upsilon_{a_-})) = c_1(\upsilon_{c_-}).
    \end{equation*}
    
    \noindent If the value $\upsilon_{c_-}$ satisfies $\upsilon_{c_-}>\upsilon_{a_+}$ it is clear that the solution $\mathcal{Y}_{u_-}$ intersects with $\mathcal{Y}_{u_+}$ at some $\upsilon_x$ such that $\upsilon_{a_+}<\upsilon_x<\upsilon_{a_-}$.
    In the opposite case that $\upsilon_{c_-}<\upsilon_{a_+}$ and the solution $\mathcal{Y}_{u_-}$ does not intersect with $\mathcal{Y}_{u_+}$ at some $\upsilon$ such that $\upsilon_{a_+}<\upsilon<\upsilon_{a_-}$, then we can decrease arbitrarily $\varepsilon$ as long as it is positive and since $\upsilon_{a_-}=\upsilon_{c_+} + \varepsilon$, then there must exists $\upsilon_{a_+}<\upsilon_x<\upsilon_{a_-}$ such that 
    \begin{equation*}
   		\mathcal{Y}_{u_+}(\upsilon_x,c_1(\upsilon_{a_+}))
   		=
   		\mathcal{Y}_{u_-}(\upsilon_x,c_2(\upsilon_{a_-})).
   	\end{equation*}

    \noindent Note now that since $\mathcal{Y}_{u_-}$ intersects with $c_1$ at $\upsilon_{b_-}$, and $\mathcal{Y}_{u_+}$ always increases but remains below $c_1$ as $\upsilon$ increases, then there must exists $\upsilon_{a_-}<\upsilon_{\max}<\upsilon_{b_-}$ such that
    \begin{equation*}
        \mathcal{Y}_{u_+}(\upsilon_{\max},c_1(\upsilon_{a_+}))
        =
        \mathcal{Y}_{u_-}(\upsilon_{\max},c_2(\upsilon_{a_-})).
    \end{equation*}
    
    \noindent Finally, by converse arguments, since $\mathcal{Y}_{u_+}$ intersects with $c_2$ at $\upsilon_{b_+}$, and $\mathcal{Y}_{u_-}$ always increases but remains below $c_2$ as $\upsilon$ decreases, then there must exists $\upsilon_{b_+}<\upsilon_{\min}<\upsilon_{a_+}$ such that
    \begin{equation*}
        \mathcal{Y}_{u_+}(\upsilon_{\min},c_1(\upsilon_{a_+}))
        =
        \mathcal{Y}_{u_-}(\upsilon_{\min},c_2(\upsilon_{a_-})).
    \end{equation*} 
\end{proof}

It should be immediately noted from Proposition \ref{prop:sequences_convergent_unique} on the accommodation property and from Proposition \ref{prop:duhem_butterfly} on the existence of an invariant butterfly loop that applying a simple periodic input $u_p\in AC(\R_+,\R)$ with only one maximum and one minimum in its periodic interval whose values are $\upsilon_{\min}$ and $\upsilon_{\max}$, then the input-output phase plot will converge to the butterfly hysteresis loop for every initial value of the output $\gamma_0\in\R$.

\subsection{First example of a Duhem butterfly operator \label{subsec:duhem_butterfly}}

As an illustrative example, we introduce now a Duhem butterfly operator which we build constructively by: i). defining arbitrary curves $c_1(\upsilon,\gamma)$ and $c_2(\upsilon,\gamma)$ satisfying conditions of Lemma \ref{lemma:sol_positive_invariance}; and ii). selecting the functions $f_1$ and $f_2$ such that these curves correspond to the zero level set (i.e. $f_1(\upsilon,c_1(\upsilon)) = f_2(\upsilon,c_2(\upsilon)) = 0$) and satisfy the hypotheses in Lemma \ref{lemma:sol_extension_outbounded} and Proposition \ref{prop:duhem_butterfly}. %We remark that we do this to illustrate how a Duhem butterfly operator can be created. 
%We note that the example presented in this sub-section 
In general, any functions $f_1$ and $f_2$ satisfying hypotheses in Lemmas \ref{lemma:sol_extension_outbounded},  \ref{lemma:sol_positive_invariance} and Proposition \ref{prop:duhem_butterfly}, which can be constructed using particular kernel functions or identified using existing models in literature, will produce Duhem butterfly operators.

Let us firstly define the curves $c_1$ and $c_2$ by
\begin{align}
    c_1(\upsilon,\gamma) &:= \phantom{+} 
        a_1 + a_2 \upsilon + a_3 \upsilon^3, \label{eq:c1_butterfly_cubic}\\
    c_2(\upsilon,\gamma) &:= 
      - b_1 - b_2 \upsilon - b_3 \upsilon^3. \label{eq:c2_butterfly_cubic}
\end{align}
In order to assign these curves as the zero level sets we can define $f_1$ and $f_2$ %the gradient functions 
as the signed vertical distance between the curve $c_1(\upsilon,\gamma)$ and %$f_1(\upsilon,\gamma)$ for 
the point $(\upsilon,\gamma)$, and respectively, between $c_1(\upsilon,\gamma)$ and the point  %$f_1(\upsilon,\gamma)$ for any 
$(\upsilon,\gamma)$. %(and $c_2$) and the evaluation point $(\upsilon,\gamma)$  of the gradient function $f_i$, for $i=1,2$, 
Here, we need to take care that the convergence conditions \eqref{eq:f1_strictly_monotonic_convergent} and \eqref{eq:f2_strictly_monotonic_convergent} are satisfied. Accordingly, we can define  $f_1$ and $f_2$ by
\begin{align}
    f_1(\upsilon,\gamma) &:= \phantom{+}\left( c_1(\upsilon)-\gamma \right) \nonumber\\
    &\phantom{:}= \phantom{+}\left( a_1 + a_2 \upsilon + a_3 \upsilon^3 - \gamma \right), \label{eq:f1_butterfly_cubic}\\
    f_2(\upsilon,\gamma) &:= - \left( c_2(\upsilon)-\gamma \right) \nonumber\\
    &\phantom{:}= \phantom{+}\left( b_1 + b_2 \upsilon + b_3 \upsilon^3 + \gamma \right). \label{eq:f2_butterfly_cubic}
\end{align}
Substituting the functions defined above into \eqref{eq:f1_strictly_monotonic_convergent} and \eqref{eq:f2_strictly_monotonic_convergent} we obtain that
\begin{equation*}
    -(\gamma_1-\gamma_2)^2<0 \quad\text{and}\quad (\gamma_1-\gamma_2)^2>0, %\label{eq:butterfly_cubic_convergenge_condition}
\end{equation*}
which are trivially satisfied.

In Fig. \ref{fig:butterfly_cubic}, we present the simulation results of a Duhem butterfly operator \eqref{eq:duhem_model} defined with \eqref{eq:f1_butterfly_cubic} and \eqref{eq:f2_butterfly_cubic} when a periodic input, whose maximum and minimum are $u_{\max}=5$ and $u_{\max}=-5$, is applied.
\begin{figure}
    \centering
    \includegraphics[width=1.0\columnwidth]{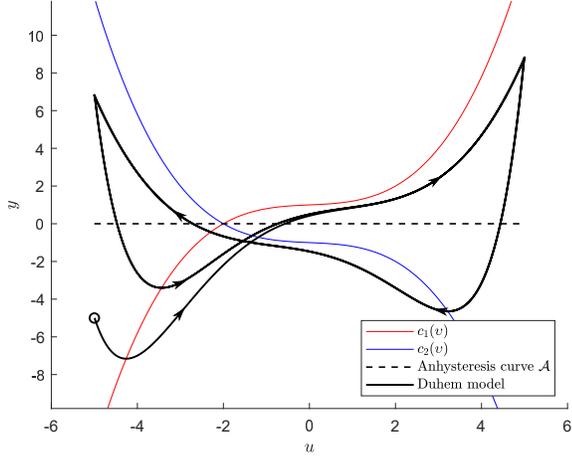}
    \caption{Butterfly hysteresis loop obtained from a model Duhem model whose gradient functions $f_1$ and $f_2$ are given by \eqref{eq:f1_butterfly_cubic} and \eqref{eq:f2_butterfly_cubic}, respectively, when a periodic input whose minimum and maximum are $\upsilon_{\min}=-5$ and $\upsilon_{\max}=5$. The initial point $(u(0),y(0))=(\upsilon_{\min},y_0)$ is marked by a circle. \label{fig:butterfly_cubic}}
\end{figure}

\subsection{Second example of Duhem butterfly operator with opposite conditions}

As remarked before, our main results in Lemma \ref{lemma:sol_positive_invariance}, Lemma \ref{lemma:sol_extension_outbounded} and Proposition \ref{prop:duhem_butterfly} hold also for the case when the signs are reversed. Correspondingly, in this subsection, we present an example of a Duhem operator that satisfy all the opposite conditions to Lemmas \ref{lemma:sol_positive_invariance}- \ref{lemma:sol_extension_outbounded} and to Proposition \ref{prop:duhem_butterfly}.   %regarding the slopes of the level set curves and the monotonicity of the gradient functions $f_1$ and $f_2$ with respect to their first variable. 

We modify slightly the previous example in Subsection \ref{subsec:duhem_butterfly} by defining $f_1$ and $f_2$ as follows.
\begin{align}
    f_1(\upsilon,\gamma) &:= \phantom{+}\left( -c_1(\upsilon)-\gamma \right) \nonumber\\
    &\phantom{:}= \phantom{+}\left( - a_1 - a_2 \upsilon - a_3 \upsilon^3 - \gamma \right), \label{eq:f1_butterfly_cubic_negative}\\
    f_2(\upsilon,\gamma) &:= - \left( -c_2(\upsilon)-\gamma \right) \nonumber\\
    &\phantom{:}= \phantom{+}\left( -b_1 - b_2 \upsilon - b_3 \upsilon^3 + \gamma \right). \label{eq:f2_butterfly_cubic_negative}
\end{align}

%\begin{align}
%    f_1(\upsilon,\gamma) &:= -
%    \left( c_1(\upsilon)-\gamma \right) \nonumber\\
%    &\phantom{:}= 
%    \left( a_1 + a_2 \upsilon + a_3 \upsilon^3 - \gamma \right), \label{eq:f1_butterfly_cubic_negative}\\
%    f_2(\upsilon,\gamma) &:= \phantom{+} 
%    \left( c_2(\upsilon)-\gamma \right) \nonumber\\
%    &\phantom{:}= 
%    \left( b_1 + b_2 \upsilon + b_3 \upsilon^3 + \gamma \right). \label{eq:f2_butterfly_cubic_negative}
%\end{align}

By vis-\'a-vis arguments to the ones of Proposition \ref{prop:duhem_butterfly}, a Duhem operator with the above $f_1$ and $f_2$ can also produce a hysteresis loops with self-intersections. This will result in the reversion of the loop orientation. % of the loops will be reversed.
Fig. \ref{fig:butterfly_cubic_negative} shows a simulation result of a Duhem butterfly operator \eqref{eq:duhem_model} defined by \eqref{eq:f1_butterfly_cubic_negative} and \eqref{eq:f2_butterfly_cubic_negative} when a periodic input, whose maximum and minimum are $u_{\max}=5$ and $u_{\min}=-5$, is applied.

\begin{figure}
    \centering
    \includegraphics[width=1.0\columnwidth]{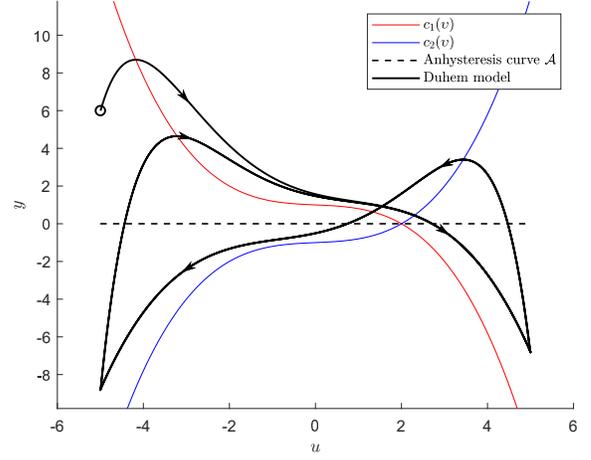}
    \caption{Butterfly hysteresis loop obtained from a Duhem model whose gradient functions $f_1$ and $f_2$ are given by \eqref{eq:f1_butterfly_cubic_negative} and \eqref{eq:f2_butterfly_cubic_negative}, respectively, when a periodic input whose minimum and maximum are $\upsilon_{\min}=-5$ and $\upsilon_{\max}=5$. The initial point $(u(0),y(0))=(\upsilon_{\min},y_0)$ is marked by a circle. \label{fig:butterfly_cubic_negative}}
\end{figure}

\subsection{A counter-example of Duhem operator with multi-loop behavior}

In this final subsection, we present %Finally, to close this section we introduce 
an example of Duhem operator whose functions $f_1$ and $f_2$ are not limited to %restricted to satisfy 
the conditions in Lemmas \ref{lemma:sol_extension_outbounded} and \ref{lemma:sol_positive_invariance} and Proposition \ref{prop:duhem_butterfly} but they satisfy the hypotheses in Proposition \ref{prop:sequences_convergent_unique}. In this example, when the Duhem operator is subjected to a periodic input signal, the input-output phase plot converges to a periodic orbit as expected and additionally the orbit can exhibit multi-loop hysteresis behavior. For constructing this example, we define the zero level set curves $c_1$ and $c_2$ by
%\begin{align}
%    c_1(\upsilon,\gamma) &:= \phantom{+} 
%        a_1 + a_2 \upsilon + a_3 \upsilon^3, \label{eq:c1_multiloop}\\
%    c_2(\upsilon,\gamma) &:= 
%      - b_1 - b_2 \upsilon - b_3 \upsilon^3, \label{eq:c2_multiloop}
%\end{align}
\begin{align}
    c_1(\upsilon,\gamma) &:= \phantom{+} 
        10\ \sin\left(\,6\pi\ \upsilon\ + \ \frac{\pi}{8}\,\right), \label{eq:c1_multiloop}\\
    c_2(\upsilon,\gamma) &:= 
        -\phantom{0}8\ \sin\left(\,6\pi\ \upsilon\ - \  \frac{\pi}{8}\,\right), \label{eq:c2_multiloop}
\end{align}
and as presented in Subsection \ref{subsec:duhem_butterfly}, the functions $f_1$ and $f_2$ are defined as the signed vertical distance between these curves (i.e., $c_1(\upsilon,\gamma)$ and $c_2(\upsilon,\gamma)$) and the point $(\upsilon,\gamma)$, respectively. Explicitly, they are given by 
\begin{align}
    f_1(\upsilon,\gamma) &:= \left( c_1(\upsilon)-\gamma \right) \nonumber\\
    &\phantom{:}= 10\ \sin\left(\,6\pi\ \upsilon\ + \ \frac{\pi}{8}\,\right) - \gamma, \label{eq:f1_multiloop}\\
    f_2(\upsilon,\gamma) &:= - \left( c_2(\upsilon)-\gamma \right) \nonumber\\
    &\phantom{:}= \phantom{0}8\ \sin\left(\,6\pi\ \upsilon\ - \  \frac{\pi}{8}\,\right) - \gamma. \label{eq:f2_multiloop}
\end{align}
The simulation results of such Duhem operator \eqref{eq:duhem_model} with $f_1$ as in  \eqref{eq:f1_multiloop} and $f_2$ as in  \eqref{eq:f2_multiloop} are shown in Fig. \ref{fig:multiloop} where multi-loop hysteresis behavior is exhibited.

\begin{figure}
    \centering
    \includegraphics[width=1.0\columnwidth]{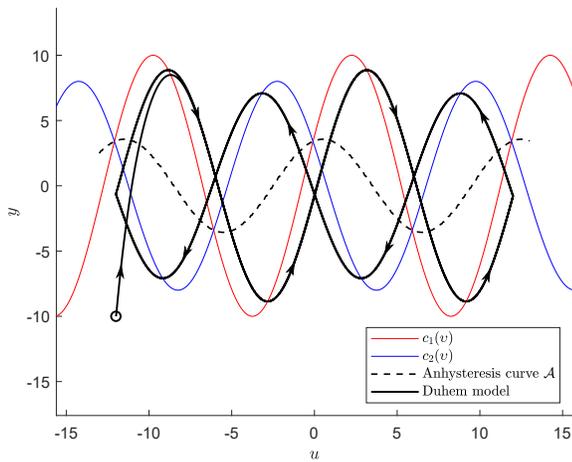}
    \caption{Multi-loop hysteresis loop obtained from a model Duhem model whose gradient functions $f_1$ and $f_2$ are given by \eqref{eq:f1_butterfly_cubic} and \eqref{eq:f2_butterfly_cubic}, respectively, when a periodic input whose minimum and maximum are $\upsilon_{\min}=-12$ and $\upsilon_{\max}=12$. The initial point $(u(0),y(0))=(\upsilon_{\min},y_0)$ is marked by a circle. \label{fig:multiloop}}
\end{figure}

%\section{Butterfly Duhem operator identification}
% \todo[inline]{Is this section going or better not?}

\section{CONCLUSIONS}\label{sec:conclusions}
In this paper we have studied and presented sufficient conditions for a class of Duhem hysteresis operators that admit butterfly loops. Firstly, we studied general conditions on the functions $f_1$ and $f_2$ so that the Duhem operator has the accommodation property. Particularly, we do not impose positive definiteness or particular form on these functions. % without restricting them to have a particular form or to be positive definite and such that the accommodation property of the Duhem operator \eqref{eq:duhem_model} holds. In other words, such that the application of a simple periodic input with a single minimum and single maximum in its periodic interval produces always a hysteresis loop.
Based on the sufficient conditions for the accommodation property, we presented sufficient conditions on $f_1$ and $f_2$ % have introduced conditions over the gradient functions 
such that the corresponding Duhem hysteresis operator is capable of exhibiting butterfly hysteresis loops. Numerical simulations show also the possibility of having multi-loop behavior when these conditions are not satisfied. The work presented in this paper can be the basis for the development of systems identification methods to model butterfly or multi-loop hysteresis phenomena in many electro-mechanical applications based on the use of integro-differential Duhem models.  

\addtolength{\textheight}{-3cm}   % This command serves to balance the column lengths

%%%%%%%%%%%%%%%%%%%%%%%%%%%%%%%%%%%%%%%%%%%%%%%%%%%%%%%%%%%%%%%%%%%%%%%%%%%%%%%%
%\section{ACKNOWLEDGMENTS}

%The authors gratefully acknowledge the contribution of National Research Organization and reviewers' comments.

%%%%%%%%%%%%%%%%%%%%%%%%%%%%%%%%%%%%%%%%%%%%%%%%%%%%%%%%%%%%%%%%%%%%%%%%%%%%%%%%

%\bibliographystyle{IEEEtran}
\bibliographystyle{IEEEtran}
\bibliography{library}

\end{document}